\documentclass{article}
\usepackage{graphicx} 
\usepackage{amsmath}
\usepackage{amsfonts}
\usepackage{amssymb}
\usepackage{bm}
\usepackage{amsthm}
\usepackage{bbm}
\usepackage{stmaryrd}
\usepackage[english]{babel}
\usepackage{dirtytalk}
\usepackage{dsfont}
\usepackage{soul}
\usepackage{xcolor}
\usepackage{cancel}
\usepackage{float}

\usepackage{hyperref}

\newtheorem{theorem}{Theorem}[section]

\newtheorem{lemma}[theorem]{Lemma}
\newtheorem{proposition}[theorem]{Proposition}
\newtheorem{definition}[theorem]{Definition}
\newtheorem{remark}[theorem]{Remark}
\newtheorem{example}[theorem]{Example}

\newtheorem{notation}[theorem]{Notation}
\newtheorem*{acknowledgement}{Acknowledgements}
\newtheorem*{address}{Address}

\numberwithin{equation}{section}

\title{Global Schauder estimates for nondivergence stationary operators modeled on homogeneous Hörmander vector fields}

\author{Matteo Faini}
\date{\begin{footnotesize} Dipartimento di Matematica, Politecnico di Milano, via Bonardi 9, 20133 Milano, Italy \end{footnotesize}}

\begin{document}
\maketitle

\begin{abstract}
In this paper we prove global regularity results and Schauder estimates for non-divergence stationary operators of the form $\mathcal{L}=\sum_{i,j} a_{ij}(x) \, X_i X_j$, where $\{X_i\}_{1 \leq i \leq m}$ are homogeneous (but not necessarily left-invariant) Hörmander vector fields in $\mathbb{R}^n$ ($n>m$), and $[a_{ij}(x)]_{i,j}$ is a symmetric uniformly positive-definite matrix with Hölder-continuous entries w.r.t. the control distance induced by the vector fields $\{X_i\}_i$. 
\end{abstract}


\vspace{5mm}

\section{Introduction}
The aim of this paper is to prove regularity results and Schauder estimates for the second order differential operator
\begin{equation}
\label{our operator L}
    \mathcal{L} = \sum_{i,j=1}^m a_{ij}(x) \, X_i X_j = a^{ij}(x) \, X_i X_j,
\end{equation}
where $X_1, \, ..., \, X_m$ are Hörmander vector fields on $\mathbb{R}^n$ ($n>m$), homogeneous w.r.t. a family of non-isotropic dilations (but not necessarily left-invariant w.r.t. any Lie group structure), and $A(x)=[a_{ij}(x)]_{1 \leq i,j \leq m}$ is a symmetric and uniformly positive-definite matrix whose coefficients are Hölder-continuous functions on $\mathbb{R}^n$ w.r.t. the control distance induced by $X_1, \, ..., \, X_m$. 

According to Einstein's convention, in (\ref{our operator L}) and whenever not specified otherwise, any repeated index appearing once as apex and once as subscript will be implicitly summed from $1$ up to $m$.

\vspace{2mm}

\hspace{-7mm} Let us give a general overview of the existing literature on this topic. After the famous paper by Hörmander \cite{Hörmander's holy graal} on hypoellipticity, a lot of work has been devoted to the analysis of Hörmander operators, and later also to more general nonvariational operators modeled on Hörmander vector fields, of the form (\ref{our operator L}). \\
In particular, a local regularity theory in both Hölder spaces and Sobolev spaces induced by the vector fields, as well as the associated estimates,
have been first proved by Folland \cite{Fo2} for a wide class of hypoelliptic operators, invariant under the action of Carnot groups; without such assumptions on the underlying algebraic structure, analogous results were proved in local form by Rothschild, Stein \cite{R.S. lifting} for general Hörmander sum of squares. Later, this local theory has been extended by Bramanti, Brandolini \cite{Hölder spaces: original vs lifted}, \cite{librone} first to evolution operators $\mathcal{H}=\mathcal{L}-\partial_t$ with $\mathcal{L}$ as in (\ref{our operator L}), and then to the wider class of operators $a^{ij}(x) X_i X_j - X_0$, with $\{X_i\}_{i=0}^m$ Hörmander vector fields and coefficients $a_{ij}$ either Hölder-continuous or satisfying a VMO condition.\\
In Sobolev spaces, global versions of these results have been then established:
\begin{itemize}
    \item in the stationary case, by Biagi, Bonfiglioli, Bramanti \cite{Global Sobolev est for SS} for homogeneous Hörmander sum of squares; later by Biagi, Bramanti \cite{Global Sobolev est with a_ij} for operators modeled on homogeneous Hörmander vector fields with $VMO$ coefficients; 
    \item in the evolution case, by Biagi, Bramanti \cite{Global Sobolev est KFP} for a class of Kolmogorov-Fokker-Planck operators with coefficients VMO in space and just measurable in time.
\end{itemize}
In Hölder spaces, instead, the only global result is the one developed by Biagi, Bramanti \cite{Hölder est Kolmogorov} for Kolmogorov-Fokker-Planck operators with coefficients Hölder-continuous in space and measurable in time, whereas a correspondent global regularity theory for stationary operators is so far lacking. The main motivation of this paper is precisely to fill this void, developing a global regularity theory parallel to \cite{Global Sobolev est for SS}, \cite{Global Sobolev est with a_ij} in the context of Hölder spaces.

\vspace{2mm}

We give a brief anticipation on the procedure that will lead us to that result.\\As usually done in this context, we will first lift $\mathcal{L}$ to a similar operator $\mathcal{\widetilde{L}}$ on a higher dimensional space, having the algebraic structure of Carnot group (see Theorem \ref{global BB lifting thm}); then prove global Schauder estimates in this richer framework, which is easier in view of the underlying group structure.
The idea is to use integral representation formulas that allow to express any smooth function and its derivatives in terms of the fundamental solution of the \say{model} operator obtained by \say{freezing} the coefficients of $\mathcal{\widetilde{L}}$ at a fixed point of $\mathbb{R}^N$. The existence of the latter, as well as its powerful properties, are the result of a deep theory about sublaplacians on Carnot groups, started by Folland \cite{Fo2} and then enriched by Bramanti, Brandolini \cite{librone}.\\
In the end of the paper we will finally remove the group structure and extend our result to the original class of operators (\ref{our operator L}), by means of the global lifting method by Biagi, Bonfiglioli \cite{lifting BB}, which is an adaptation of the classical by Folland \cite{Fo1}; we will also exploit some relations between Hölder spaces in the original setting and in the lifted framework, which we owe to \cite{Hölder spaces: original vs lifted}.

\vspace{2mm}

\hspace{-7mm} Let us clarify the theoretical setting. We assume that $X_1, \, ..., \, X_m$ are smooth vector fields on $\mathbb{R}^n$ $(n>m)$ with the following properties:
\begin{itemize}
    \item [(H1)] $X_1, \, ...,\, X_m$ are $1$-homogeneous w.r.t. a family of non-isotropic dilations $\{\delta_{\lambda}\}_{\lambda>0}$ of $\mathbb{R}^n$, which act as
    \begin{equation*}
        \delta_{\lambda}(x) = (\lambda^{\sigma_1} x_1, \, ..., \, \lambda^{\sigma_n} x_n) \hspace{5mm} \text{for suitable integers} \: 1 = \sigma_1 \leq ... \leq \sigma_n.
    \end{equation*}
    Explicitly, this means that for every smooth function $f$ it holds
    \begin{equation*}
        X_i (f \circ \delta_{\lambda}) = \lambda (X_i f) \circ \delta_{\lambda} \quad \forall 1 \leq i \leq m, \, \lambda>0.
    \end{equation*}
    \item [(H2)] $X_1, \, ..., \, X_m$ are linearly independent (as vector fields) and satisfy Hörmander's rank condition in the origin of $\mathbb{R}^N$.
\end{itemize}

\vspace{2mm}

\begin{remark}
    (H1) and (H2) imply that
    Hörmander's rank condition for $\{X_i\}_{i=1}^m$ actually holds in the whole space $\mathbb{R}^n$ (see \cite[Remark 3.2]{global est fund sol hom Horm op} for more details).
\end{remark}

\vspace{2mm}

Keeping the notation above, we denote by 
\begin{equation*}
    q := \sum_{i=1}^n \sigma_i
\end{equation*} 
the \emph{homogeneous dimension} of $\mathbb{R}^n$ w.r.t. $\delta_{\lambda}$, which we assume to be strictly greater than $2$. We point out that this is not an actual restriction, as the only case ruled out is when $\mathcal{L}$ is a uniformly elliptic operator in the Euclidean space $\mathbb{R}$ or $\mathbb{R}^2$, for which global regularity results are already well known.\\
Let $d_X$ denote the \emph{control distance} induced by $X_1, \, ..., \, X_m$ on $\mathbb{R}^n$ (see Section \ref{sec 2} for the precise definition).
\\For $\Lambda \in (0,1)$, let $\mathcal{M}_{\Lambda}$ be the set of symmetric constant matrices defined as
\begin{equation*}
    \mathcal{M}_{\Lambda} = \{B=[b_{ij}] \in \mathbb{R}^{m \times m}: B=B^T, \, \Lambda |\xi|^2 \leq  b^{ij}\, \xi_i \xi_j \leq \Lambda^{-1} |\xi|^2 \:\, \forall \xi \in \mathbb{R}^m\}.
\end{equation*}

\hspace{-7mm} Let now $\Lambda \in (0,1), \, \alpha \in (0,1)$ be fixed parameters. Consider a matrix 
\begin{itemize}
    \item [(H3)] $[a_{ij}(x)] \in \mathcal{M}_{\Lambda}$ for all $x \in \mathbb{R}^n$ with entries $a_{ij}$ satisfying 
    \begin{equation*}
        |a_{ij}(x)-a_{ij}(x')| \leq c \, d_X(x,x')^{\alpha} \quad \forall x,x' \in \mathbb{R}^n, \, 1 \leq i,j \leq m
    \end{equation*}
    for some constant $c>0$.
\end{itemize}
To keep track of the dependence of the constants, for any $k \geq 0$ we denote
\begin{equation*}
    ||a||_{k,X}:=\max_{1 \leq i,\, j \leq m} ||a_{ij}||_{C^{k,\alpha}_X(\mathbb{R}^n)},
\end{equation*}
whenever the latter is well defined, i.e. provided that $\{a_{ij}\}_{i,j=1}^m \subset C^{k,\alpha}_X(\mathbb{R}^n)$ (see Section \ref{sec 2} for the definition of the spaces $C^{k,\alpha}_X(\cdot)$); when there is no risk of confusion, we will drop the \say{$X$} and write just $||a||_k$.

Under these assumptions, the operator (\ref{our operator L}) is degenerate elliptic, but the missing directions are \say{recovered} by commutators, in the sense of \cite{Hörmander's holy graal}. Note, however, that in view of the fact that its coefficients $a_{ij}$ are not smooth, the operator $\mathcal{L}$ is not hypoelliptic. For this kind of operators, we want to prove the following:

\vspace{2mm}

\begin{theorem}[Global Schauder estimates]
\label{global C^alpha estimate thm}
    Under assumptions (H1)-(H2)-(H3) on $\mathcal{L}$, there exists a positive constant $c_0=c_0( \{X_i\}_{i=1}^m,\Lambda, \alpha, ||a||_{0,X})$ such that for every $u \in C^{2,\alpha}_X(\mathbb{R}^n)$ the following estimate holds:
    \begin{equation}
    \label{global C^alpha estimate}
        ||u||_{C^{2,\alpha}_X(\mathbb{R}^n)} \leq c_0 \, \{||\mathcal{L}u||_{C^{\alpha}_X(\mathbb{R}^n)} + ||u||_{L^{\infty}(\mathbb{R}^n)} \}.
    \end{equation}
    If moreover we assume that, for some $k \geq 1$, $a_{ij} \in C^{k,\alpha}_X(\mathbb{R}^n)$ for all $1 \leq i,j \leq m$, then there exists a constant $c_k=c(k,\{X_i\}_{i=1}^m, \Lambda, \alpha, ||a||_{k,X})>0$ such that, for every $u \in C^{2,\alpha}_X(\mathbb{R}^n)$ satisfying $\mathcal{L} u \in C^{k,\alpha}_X(\mathbb{R}^N)$, there hold $u \in C^{k+2,\alpha}_X(\mathbb{R}^n)$ and
    \begin{equation*}
        ||u||_{C^{k+2,\alpha}_X(\mathbb{R}^n)} \leq c_k \left\{ ||\mathcal{L} u||_{C^{k,\alpha}_X(\mathbb{R}^n)} + ||u||_{L^{\infty}(\mathbb{R}^n)} \right\}.
    \end{equation*}
\end{theorem}

\vspace{2mm}

We stress that in the above theorem (which is the main result of the present paper) we are just requiring a homogeneity property for the Hörmander vector fields $X_1, \, ..., \, X_m$ to hold, without however assuming the presence of any underlying group structure.

\vspace{2mm}

Let $\mathfrak{a}=Lie\{X_1, \, ..., \, X_m\}$ denote the smallest Lie algebra containing $\{X_i\}_{i=1}^m$. Due to the homogeneous structure of $(\mathbb{R}^n, \delta_{\lambda})$ we have that $\mathfrak{a}$ is finite dimensional (as a real vector space), so we can set $N:=dim(\mathfrak{a})$.
Note that (H2) implies $N \geq n$, but in this paper we always assume $N>n$, ruling out just the trivial case where there already exists a group of translations that make $X_1, \, ..., \, X_m$ left-invariant (as noted in \cite[Section 2]{Gauss bounds with a_ij (BB)}).
Called $p:=N-n\geq 1$, we denote the points of the lifted space $\mathbb{R}^N$ by $(x,\xi)$, with $x \in \mathbb{R}^n$ and $\xi \in \mathbb{R}^p$.

\vspace{2mm}

\hspace{-7mm} We summarize the main results of the global lifting by Biagi, Bonfiglioli \cite{lifting BB}:

\vspace{2mm}

\begin{theorem}[Global lifting]
\label{global BB lifting thm}
    Under the above assumptions (of which we inherit the notation), there exists a Carnot group $\mathbb{G}=(\mathbb{R}^N, *, D_{\lambda})$ of step $\sigma_n$ with $m$ generators $\widetilde{X}_1, \, ..., \, \widetilde{X}_m$, which are related to $X_1, \, ..., \, X_m$ as follows:
\begin{equation}
\label{relation X_i tilde and X_i}
    \widetilde{X}_i = X_i + \sum_{j=n+1}^N R_j(x,\xi) \frac{\partial}{\partial \xi_j},
\end{equation}
for suitable smooth functions $\{R_j\}_{j=n+1}^N$. Moreover, there exist integers $1 \leq \tau_1 \leq ... \leq \tau_p$ such that the \say{lifted} dilations $D_{\lambda}$ act as
\begin{equation*}
    D_{\lambda}(x,\xi)=(\delta_{\lambda}(x), \delta_{\lambda}^*(\xi)), \quad \text{with} \: \delta_{\lambda}^*(\xi)=(\lambda^{\tau_1} \xi_1, \, ..., \, \lambda^{\tau_p} \xi_p).
\end{equation*}
\end{theorem}

\vspace{2mm}

\hspace{-7mm} For the proof see \cite[Theorem 3.2]{lifting BB}.

\vspace{2mm}

\begin{remark}[Comparison with Rothschild-Stein's lifting]
    Let us point out some crucial differences between the lifting in Theorem \ref{global BB lifting thm} and the first famous lifting result, proved by Rothschild and Stein in \cite{R.S. lifting}. \\The latter states that \emph{any} system of Hörmander vector fields can be lifted on a higher dimensional space to a system of free Hörmander vector fields, which, in turn, are locally approximated (in the sense of Taylor series) by the generators of a Carnot group; i.e., it is a very general lifting result, but holding only in local form and consisting in a two-step approximation. \\Conversely, Folland's lifting (of which Theorem \ref{global BB lifting thm} is a useful corollary) directly lifts a system of homogeneous Hörmander vector fields to a set of generators of a Carnot group, and that lifting is performed globally; this is made possible by the homogeneous structure, which makes the above lifting less general than Rothschild-Steins's one, but also much more powerful.
    \end{remark}

    \vspace{2mm}

Defined $\widetilde{a}_{ij}(x,\xi):=a_{ij}(x)$ for every $(x,\xi) \in \mathbb{R}^N$, the idea of the paper is to prove a result analogous to Theorem \ref{global C^alpha estimate thm} for the \say{lifted} operator
\begin{equation}
\label{def L tilde}
    \mathcal{\widetilde{L}} = \widetilde{a}^{ij}(x,\xi) \, \widetilde{X}_i \widetilde{X}_j,
\end{equation}
which is much nicer thanks to the presence of an underlying group structure; once we will have proved this, in Section \ref{sec 5} we will remove the group structure and come back to our original setting by means of (\ref{relation X_i tilde and X_i}) and some relations between original and lifted Hölder spaces.
\\Let us make a basic example of how the lifting procedure actually works:

\vspace{2mm}

\begin{example}
    Let us consider $\mathbb{R}^2 \ni (x,y)$, endowed with the dilations $\delta_{\lambda}(x,y)=(\lambda x, \lambda^2 y)$. The vector fields
    \begin{equation*}
        X_1 = \partial_x, \quad X_2 = x \partial_y
    \end{equation*}
    are $1$-homogeneous w.r.t. $\delta_{\lambda}$ and satisfy Hörmander's rank condition in $\mathbb{R}^2$, since $X_1=\partial_x$ and $[X_1,X_2]=\partial_y$ obviously span $\mathbb{R}^2$ at any point.\\
    We lift $X_1, \, X_2$ to new vector fields $\widetilde{X}_1, \, \widetilde{X}_2$ defined on $\mathbb{R}^3 \ni (x,y,\xi)$ as follows:
    \begin{align*}
        & \widetilde{X}_1=X_1=\partial_x, \quad \widetilde{X}_2=X_2+\partial_{\xi} = x \partial_y + \partial_{\xi}
    \end{align*}
    and we introduce the following translations and dilations on $\mathbb{R}^3$:
    \begin{align*}
        & (x,y,\xi)*(x',y',\xi'):=(x+x', y+y'+x\xi'-x'\xi,\xi+\xi');\\
        & D_{\lambda}(x,y,\xi)=(\lambda x, \lambda^2 y, \lambda \xi).
    \end{align*}
    It is well known that $(\mathbb{R}^3, *, D_{\lambda})$ is a Carnot group (the so called \emph{Heisenberg group} $\mathbb{H}^1$) and $\widetilde{X}_1, \, \widetilde{X}_2$ are its generators, so $X_1, \, X_2$ have been lifted to generators of a Carnot group on a strictly higher dimensional space; here \say{lifted} means that the actions of $X_i$ and $\widetilde{X}_i$ on any function $g \in C^{\infty}(\mathbb{R}^3)$ that do not depend on the added variable $\xi$ coincide.
\end{example}

\vspace{2mm}

As already mentioned above, the main step in the proof of Theorem \ref{global C^alpha estimate thm} consists in proving an analogous result in the context of Carnot groups, namely for operators (\ref{def L tilde}). Since this result may have an independent interest, we state it explicitly:

\vspace{2mm}

\begin{theorem}[Global Schauder estimates in Carnot groups]
\label{1.4: Schauder est Carnot thm}
    Let $\mathbb{G}$ be a Carnot group in $\mathbb{R}^N$, with generators $Y_1, \, ..., \, Y_m$ and homogeneous dimension $Q>2$. Fixed $\alpha \in (0,1), \, \Lambda \in (0,1)$, consider a matrix $[a_{ij}(z)] \in \mathcal{M}_{\Lambda}$ for all $z \in \mathbb{R}^N$, with coefficients $a_{ij} \in C^{\alpha}_Y(\mathbb{R}^N)$ for all $1 \leq i,j \leq m$, then let $L = a^{ij}(z) \, Y_i Y_j$.
    \\
    There exists a constant $c_0=c_0(\mathbb{G}, \Lambda, \alpha, ||a||_{0,Y})>0$ such that the inequality
    \begin{equation}
    \label{Schauder est Carnot eqn}
        ||u||_{C^{2,\alpha}_Y(\mathbb{R}^N)} \leq c_0 \left\{ ||L u||_{C^{\alpha}_Y(\mathbb{R}^N)} + ||u||_{L^{\infty}(\mathbb{R}^N)} \right\}
    \end{equation}
    holds for any function $u \in C^{2,\alpha}_Y(\mathbb{R}^N)$.
\end{theorem}

\vspace{2mm}

\hspace{-7mm} We will also prove the following version for higher order derivatives:

\vspace{2mm}

\begin{theorem}[Higher order estimates in Carnot groups]
    \label{global est higher derivs thm}
    Under the assumptions of Theorem \ref{1.4: Schauder est Carnot thm} and further requiring that, for some $k \geq 1$, $a_{ij} \in C^{k,\alpha}_Y(\mathbb{R}^N)$ for all $1 \leq i, \, j \leq m$, there exists a constant $c_k=c(k,\mathbb{G}, \Lambda, \alpha, ||a||_{k,Y})>0$ s.t., for every $u \in C^{2,\alpha}_Y(\mathbb{R}^N)$ with $Lu \in C^{k,\alpha}_Y(\mathbb{R}^N)$, there hold $u \in C^{k+2,\alpha}_Y(\mathbb{R}^N)$ and
    \begin{equation}
        \label{est high order derivs ineq} ||u||_{C^{k+2,\alpha}_Y(\mathbb{R}^N)} \leq c_k \left\{ ||Lu||_{C^{k,\alpha}_Y(\mathbb{R}^N)} + ||u||_{L^{\infty}(\mathbb{R}^N)}\right\}.
    \end{equation}
\end{theorem}

\vspace{2mm}

The structure of the paper is as follows. In Sections \ref{sec 2}-\ref{sec 4} we always work under the further assumption that a Carnot group structure is available, namely we build the tools for the proof of Theorems \ref{1.4: Schauder est Carnot thm} and \ref{global est higher derivs thm}. In particular, in section \ref{sec 2} we introduce basic notions and known results about Carnot groups, as well as the associated notation; we also collect some lemmas that will come at hand repeatedly during the procedure. In Section \ref{sec 3} we prove a local version of the desired Schauder estimate in Carnot groups, which is then turned into (\ref{Schauder est Carnot eqn}) through interpolation inequalities and cutoff functions. In section \ref{sec 4}, suitable representation formulas for derivatives of smooth functions are deduced, and then exploited to obtain Hölder estimates for higher order derivatives. Finally, in section \ref{sec 5} we extend all the above results to our original setting (that is, when no group of translations is available) by means of the abovementioned lifting method, culminating with the proof of Theorem \ref{global C^alpha estimate thm}.

\vspace{5mm}

\begin{acknowledgement}
    The author would like to thank his PhD supervisor, Marco Bramanti, for introducing him to this topic, sharing relevant background material, and for valuable advice throughout the development of this work.
\end{acknowledgement}

\vspace{5mm}

\section{Preliminaries on Carnot groups}
\label{sec 2}
From now on until the end of Section \ref{sec 4}, we assume that $\mathbb{R}^N$ is endowed with a Carnot group structure $\mathbb{G}=(\mathbb{R}^N, *, D_{\lambda})$, namely
\begin{itemize}
    \item [(i)] $*$ is a Lie group law on $\mathbb{R}^N$ (to which we refer as \say{translation}),
    \item [(ii)] $\{D_{\lambda}\}_{\lambda>0}$ is a family of automorphisms of $(\mathbb{R}^N,*)$ (called \say{dilations}), acting as
    \begin{equation*}
        D_{\lambda}(x)=(\lambda^{\sigma_1} x_1,..., \lambda^{\sigma_N} x_N) \quad \text{for suitable integers} \: 1 = \sigma_1 \leq ... \leq \sigma_N.
    \end{equation*}
    \item [(iii)] the Lie algebra $\mathfrak{a}$ of left-invariant vector fields on $(\mathbb{R}^N,*)$ admits a stratification $\mathfrak{a}=\mathfrak{a}_1 \oplus \mathfrak{a}_2 \oplus \, ... \, \oplus \mathfrak{a}_s$, where: $\mathfrak{a}_1$ is the span of $m$ linearly independent vector fields $Y_1, \, ..., \, Y_m$, which are $1$-homogeneous w.r.t. $\{D_{\lambda}\}_{\lambda}$; $\mathfrak{a}_k = [\mathfrak{a}_1, \mathfrak{a}_{k-1}]$ for every $2 \leq k \leq s$ and $[\mathfrak{a}_1, \mathfrak{a}_s]=0$.
\end{itemize}
We recall that condition (iii) can be equivalently stated by asking that there exist smooth vector fields $Y_1,..., Y_m$ which are left-invariant w.r.t. $*$, $1$-homogeneous w.r.t. $\{D_{\lambda}\}_{\lambda}$ and satisfy Hörmander condition in $\mathbb{R}^N$ (that is, together with their commutators of arbitrary order, they span $\mathbb{R}^N$ at any point of $\mathbb{R}^N$ itself).
The vector fields $Y_1, \, ..., \, Y_m$ are called \emph{(Lie) generators} of $\mathbb{G}$, whereas $s \geq 1$ is called \emph{step} of $\mathbb{G}$ and stands for the length of commutators of $Y_1, \, ..., \, Y_m$ that are required to fulfill Hörmander condition at any point. 
We also denote by
\begin{equation*}
    Q = \sum_{i=1}^N \sigma_i
\end{equation*}
the \emph{homogeneous dimension} of $\mathbb{G}$, which we assume to be strictly larger than $2$.

Let also $Y_{m+1},...,Y_N$ be smooth vector fields on $\mathbb{R}^N$ such that $\mathcal{A}=\{Y_1,...,Y_N\}$ is a basis of $\mathfrak{a}$ with the following structure:
\begin{equation*}
    \mathcal{A}=\mathcal{A}_1 \cup \, ... \, \cup \mathcal{A}_s: \: \mathcal{A}_1=\{Y_1, \, ..., \, Y_m\}, \: \mathcal{A}_i \: \text{basis of} \: \mathfrak{a}_i \: \forall 1 \leq i \leq s.
\end{equation*}
For all $1 \leq i \leq N$, we denote by $Y_i^R$ the (unique) right-invariant vector field agreeing with $Y_i$ at the origin of $\mathbb{R}^N$, which is clearly homogeneous of the same degree as $Y_i$, namely $\sigma_i$. Given any multindex $I=(i_1, \, i_2, \, ..., \, i_h)$ with $i_1, \, i_2, \, ..., \, i_h \in \{1, \, 2, \, ..., \, N\}$, we denote by $Y_I$ the differential operator $Y_{i_1} \, Y_{i_2} \,... \,Y_{i_h}$; analogously, $Y_I^R$ stands for $Y_{i_1}^R \, ... \, Y_{i_h}^R$. According to the usual notation, we also denote by $l(I)=h$ the length of $I$ and by $|I|=\sigma_{i_1}+...+\sigma_{i_h}$ its \say{weight} (that is, the degree of homogeneity of the operator $Y_I$); we point out that, whenever a multindex $I$ has entries only in $\{1,...,m\}$ (which will almost always be our case) the notions of length and weight coincide, thus we will just employ the symbol $|I|$. For homogeneity of notation, when convenient we will use the symbol $Y_I$ also in the case $|I|=0$, implicitly meaning that $Y_I$ is the identity operator, i.e. $Y_I u \equiv u$ for any function $u$.

We recall some useful results based on invariance and homogeneity:

\vspace{2mm}

\begin{proposition}
\label{prop: translations and dilations}
    For all $1 \leq i \leq N$ and functions $f, \, g: \mathbb{R}^N \to \mathbb{R}$ for which the following integrals converge absolutely, there holds:
    \begin{equation}
    \label{IPP left-right invariant vf}
        \int_{\mathbb{R}^N} (Y_i^R f)(y^{-1} * x) \, g(y) \, dy = \int_{\mathbb{R}^N}f(y^{-1} * x) \, (Y_i g)(y) \, dy.
    \end{equation}
    Moreover, if $f$ is a smooth function, homogeneous of degree $\beta \geq 0$, and $j \in \{1, \, ..., \, N\}$ is such that $\sigma_j>\beta$, then $Y_j$ and $f$ commute, namely
         \begin{equation}
         \label{Y_j and r_ij commute by homogeneity}
             Y_j(fg)=f \cdot Y_jg \hspace{5mm} \forall g \in C^{\infty}(\mathbb{R}^N).
         \end{equation}
\end{proposition}

\vspace{2mm}


\begin{proof}
    For the proof of (\ref{IPP left-right invariant vf}) see \cite[Proposition 3.47]{librone}. As for (\ref{Y_j and r_ij commute by homogeneity}), note that for all $k \in \{1,...,N\}$ with $\sigma_k \geq \sigma_j$ the function $\partial_{x_k} f$ is smooth and homogeneous of degree $\beta-\sigma_k<0$, thus it identically vanishes (otherwise it would be unbounded at the origin, contradicting the smoothness assumption on $f$). On the other side, denoting $Y_j=\sum_{i=1}^N b_{ij}(x) \, \partial_{x_i}$, the $\sigma_j$-homogeneity of $Y_j$ implies the $(\sigma_i-\sigma_j)$-homogeneity of $b_{ij}$ (for any $i$), which in turn yields $b_{ij} \equiv 0$ for all $i \in \{1,...,N\}$ s.t. $\sigma_i<\sigma_j$. Therefore $b_{ij} \not\equiv 0$ only if $\sigma_i \geq \sigma_j>\beta$, so
    \begin{equation*}
        (Y_j f)(x) = \sum_{i:\, \sigma_i \geq \sigma_j} b_{ij}(x) \, \partial_{x_i}f(x)=0 \quad \forall x \in \mathbb{R}^N,
    \end{equation*}
    and the thesis follows by Leibniz rule.
\end{proof}

\vspace{2mm}

For the sake of developing a Schauder theory in Hölder spaces, we must define suitable Hölder spaces that are coherent with our group structure. To this aim, we recall that a Carnot group can be \say{naturally} endowed with a distance, that is the so called \emph{control distance} induced by the generators $Y_1,...,Y_m$, of which we now briefly give the definition. 

\vspace{2mm}

\begin{definition}
    Denoted by $C_{x,y}(\delta)$ the set of absolutely continuous curves $\gamma:[0,1]\to \mathbb{R}^N$ with $\gamma(0)=x, \, \gamma(1)=y$ and
\begin{equation*}
     \gamma'(t)=\alpha^i(t) \, (Y_i)_{\gamma(t)} \hspace{5mm} \text{for a.e. } t \in (0,1), \, \text{with} \, ||\alpha^i||_{L^{\infty}(0,1)}<\delta \:\,\, \forall 1 \leq i \leq m
\end{equation*}
(for any $x,y \in \mathbb{R}^N, \, \delta>0$), the \emph{control distance} (also called \emph{Carnot-Caratheodory distance}, or in short \emph{CC distance}) induced by $Y_1,...,Y_m$ is defined as the map $d_Y(x,y):=\inf\{\delta>0: C_{x,y}(\delta) \neq \emptyset\}$ for all $x,y \in \mathbb{R}^N$.
\end{definition}

\vspace{2mm}

Whenever $Y_1,...,Y_m$ satisfy Hörmander condition in the whole space, it is well known (see \cite[Theorems 1.45-3.54]{librone}) that $d_Y$ is indeed a distance on $\mathbb{R}^N$, as well as that it is left-invariant w.r.t. translations induced by $*$ and $1$-homogeneous w.r.t. dilations $D_{\lambda}$ (i.e., coherent with the group structure), namely:
\begin{align}
    & \label{d_Y is left inv} d_Y(z * x, z * y) = d_Y(x,y) \hspace{5mm} \forall x, \, y, \, z \in \mathbb{R}^N;
    \\
    & \label{d_Y is homogeneous}
    d_Y(D_{\lambda}(x), D_{\lambda}(y)) = \lambda \, d_Y(x,y) \hspace{5mm} \forall x, \, y \in \mathbb{R}^N, \, \lambda>0.
\end{align}
It is straightforward that the induced map $||x||_Y:=d_Y(x,0)$ is a \emph{homogeneous norm} on $\mathbb{G}$, according to the following definition:

\vspace{2mm}

\begin{definition}
\label{def homogeneous norm}
    A \emph{homogeneous norm} on the Carnot group $\mathbb{G}$ is a continuous map $||\cdot || : \mathbb{R}^N \to [0,+\infty)$ satisfying the following properties:
    \begin{itemize}
        \item[(i)] $||x||=0 \iff x=0$;
        \item [(ii)] $||D_{\lambda}(x)||=\lambda ||x||$ for every $x \in \mathbb{R}^N, \, \lambda>0$;
        \item [(iii)] there exists a constant $c \geq 1$ such that $||x|| \leq c||x^{-1}||$ for all $x \in \mathbb{R}^N$;
        \item [(iv)] there exists a constant $\kappa \geq 1$ for which $||x*y|| \leq \kappa(||x||+||y||)$ for all $x,y \in \mathbb{R}^N$ (\emph{quasi-triangle inequality}).
    \end{itemize}
\end{definition}

\vspace{2mm}

\hspace{-7mm}
 We point out, however, that the CC distance $d_Y$ (and hence also $||\cdot||_Y$) need not be smooth, which could lead to regularity issues. For this reason, we introduce a smooth homogeneous norm equivalent to the latter:
\begin{equation}
\label{smooth homogeneous norm}
    ||\cdot|| : \mathbb{R}^N \to [0,+\infty), \quad ||x||=\left(  \sum_{i=1}^{N}|x_{i}|^{2\sigma_{N}!/\sigma_{i}}\right)  ^{1/(2\sigma_{N}!)},
\end{equation}
where the $\sigma_i$'s are the exponents of the dilations $\{D_{\lambda}\}_{\lambda>0}$. \\It is immediate to check that the above is indeed a homogeneous norm, smooth outside the origin. The equivalence of $||\cdot||_Y$ and $||\cdot||$ follows by the equivalence of any couple of homogeneous norms on $\mathbb{R}^N$ (see \cite[Theorem 3.12]{librone}).\\
We will perform most of our computations w.r.t. the norm $||\cdot||$, and then \say{translate} the results in terms of the CC norm and distance just by equivalence. The \emph{quasisymmetric quasidistance} induced by $||\cdot||$ and the associated balls will be respectively denoted by $d(\cdot,\cdot)$, which acts as 
\begin{equation*}
    d(x,y):=||y^{-1}*x||, \quad x,y \in \mathbb{R}^N,
\end{equation*}
and $B_r(x_0)$ (for any $r>0$ and $x_0 \in \mathbb{R}^N$).
We recall that $d(\cdot,\cdot)$ has the following properties:
\begin{itemize}
    \item [(i')] $d(x,y) \geq 0$ for every $x, \, y \in \mathbb{R}^N$ and $d(x,y)=0 \iff x=y$;
    \item [(ii')] $d(D_{\lambda}(x), D_{\lambda}(y))=\lambda d(x,y)$ for all $x, \, y \in \mathbb{R}^N, \, \lambda>0$;
    \item [(iii')] there exists $c \geq 1$ such that $d(x,y) \leq c d(y,x)$ for all $x, \, y \in \mathbb{R}^N$;
    \item [(iv')] there exists a constant $\kappa \geq 1$ such that $d(x,y) \leq \kappa (d(x,z)+d(z,y))$ for all $x, \, y, \, z \in \mathbb{R}^N$ (\emph{quasi-triangle inequality}).
\end{itemize}
((i')-(iii')-(iv') being the usual definition of quasisymmetric quasidistance, (ii') following directly by homogeneity of $||\cdot||$).

\vspace{2mm}

\hspace{-7mm} We recall a sort of \say{Lagrange theorem} holding true in Carnot groups, w.r.t. any homogeneous quasisymmetric quasidistance:

\vspace{2mm}

\begin{lemma} \label{UGE Lemma 7.6}
    There exists a constant $\overline{c}=\overline{c}(\mathbb{G})>0$ such that for every $u \in C^1_Y(\mathbb{R}^N)$ and every $x, \, x' \in \mathbb{R}^N$ there holds
    \begin{equation*}
        |u(x)-u(x')| \leq \overline{c} \, d(x,x') \, \sup_{d(x',x'') \leq \overline{c}d(x,x')} \sum_{i=1}^m |Y_i u(x'')|.
    \end{equation*}
\end{lemma}

\vspace{2mm}

\hspace{-7mm} For the proof see \cite[Lemma 7.6]{UGE}.

\vspace{2mm}

\hspace{-7mm} We now introduce Hölder spaces induced by our generators $Y_1,...,Y_m$.

\vspace{2mm}

\begin{definition}
\label{def Hölder space}
    Let $\Omega \subseteq \mathbb{R}^N$, $\alpha \in (0,1)$ and $k \geq 1$ be an integer. We denote by $C^{\alpha}_Y(\Omega) \equiv C^{0,\alpha}_Y(\Omega)$ the set of bounded functions $u$ over $\Omega$ that satisfy
    \begin{equation*}
        |u|_{C^{\alpha}_Y(\Omega)} := \sup_{x,\, x' \in \Omega, \, x \neq x'} \frac{|u(x)-u(x')|}{d_Y(x,x')^{\alpha}} < +\infty,
    \end{equation*}
    and by $C^{k,\alpha}_Y(\Omega)$ the set of functions $u \in C^{\alpha}_Y(\Omega)$ such that, for every $|I| \leq k$, the directional derivative $Y_I u$ exists (pointwise) and belongs to $C^{\alpha}_Y(\Omega)$.
    For such functions, we also denote:
    \begin{align*}
        & ||u||_{C^{\alpha}_Y(\Omega)} := ||u||_{L^{\infty}(\Omega)} + |u|_{C^{\alpha}_Y(\Omega)}; \\
        & ||u||_{C^{k,\alpha}_Y(\Omega)}:= ||u||_{C^{\alpha}_Y(\Omega)} + \sum_{1 \leq |I| \leq k} ||Y_I u||_{C^{\alpha}_Y(\Omega)}.
    \end{align*}
    Finally, the notation $C^{k,\alpha}_{0,Y}(\Omega)$ will stand for the space of functions  $u \in C^{k,\alpha}_Y(\Omega)$ compactly supported in $\Omega$.
\end{definition}

\vspace{2mm}

\begin{notation}
    When convenient, we will use the following compact notation:
    \begin{align*}
        & ||D^k u||_{L^{\infty}(\Omega)} := \sum_{|I|=k} ||Y_I u||_{L^{\infty}(\Omega)}; \quad |D^k u|_{C^{\alpha}_Y(\Omega)} := \sum_{|I|=k} |Y_I u|_{C^{\alpha}_Y(\Omega)}; \\
        & ||D^k u||_{C^{\alpha}_Y(\Omega)} := \sum_{|I|=k} ||Y_I u||_{C^{\alpha}_Y(\Omega)}
    \end{align*}
    (for any $k \geq 1$, provided all the above quantities are well-defined).
\end{notation}

\vspace{2mm}

Fixed $\Lambda \in (0,1)$ and $\alpha \in (0,1)$, let $A(x)=[a_{ij}(x)] \in \mathcal{M}_{\Lambda}$ for all $x \in  \mathbb{R}^N$ with entries $a_{ij} \in C^{\alpha}_Y(\mathbb{R}^N)$ for any $1 \leq i,j \leq m$; we denote
\begin{equation*}
    ||a||_{0,Y} := \max_{1 \leq i,\, j \leq m} ||a_{ij}||_{C^{\alpha}_Y(\mathbb{R}^N)}; \quad ||a||_{k,Y} := \max_{1 \leq i,\, j \leq m} ||a_{ij}||_{C^{k,\alpha}_Y(\mathbb{R}^N)},\: k \geq 1
\end{equation*}
(whenever the above quantities are well defined).
We focus on the operator
\begin{equation}
\label{our operator L_Y Carnot}
    L=a^{ij}(x) \, Y_i Y_j.
\end{equation}

\vspace{2mm}

For any $k \geq 0$, it is well known $C^{k,\alpha}_Y(\Omega)$ is a vector space and moreover, given $f, \, g \in C^{k,\alpha}_Y(\Omega)$, there holds $fg \in C^{k,\alpha}_Y(\Omega)$ and
\begin{equation}
\label{Hölder norm of a product}
    ||fg||_{C^{k,\alpha}_Y(\Omega)} \leq ||f||_{C^{k,\alpha}_Y(\Omega)} ||g||_{C^{k,\alpha}_Y(\Omega)}.
\end{equation}
An immediate consequence is that $Lu \in C^{\alpha}_Y(\Omega)$ whenever $u \in C^{2,\alpha}_Y(\Omega)$. Furthermore, for any $u \in C^{\alpha}_Y(B_R(\eta))$ vanishing at a point $z \in B_R(\eta)$, it holds
\begin{equation*}
        |u(x)|=|u(x)-u(z)| \leq c \, d(x,z)^{\alpha} |u|_{C^{\alpha}_Y(B_R(\eta))} \leq c \, R^{\alpha} |u|_{C^{\alpha}_Y(B_R(\eta))},
    \end{equation*}
    and so, for whatever fixed $R_0$ and any $0<R \leq R_0$:
    \begin{align}
    \label{Hölder norm vs seminorm for u that vanishes at a point}
        ||u||_{L^{\infty}(B_R(\eta))} \leq cR^{\alpha} |u|_{{C^{\alpha}_Y(B_R(\eta))}}; \quad ||u||_{C^{\alpha}_Y(B_R(\eta))} \leq c'(R_0)|u|_{C^{\alpha}_Y(B_R(\eta))}.
    \end{align}

    \vspace{2mm}





\vspace{2mm}

Given $\overline{x} \in \mathbb{R}^N$, let us consider the following operator with constant coefficients $a_{ij}$ (obtained by \say{freezing} the coefficients $a_{ij}(\cdot)$ of $L$ in $\overline{x}$):
\begin{align}
\label{our freezed operator Carnot}
    & L_{\overline{x}} = a^{ij}(\overline{x}) \, Y_i Y_j.
\end{align}
It is immediate that also $L_{\overline{x}}u \in C^{\alpha}_Y(\Omega)$ for every $\Omega \subset \mathbb{R}^N, \, u \in C^{2,\alpha}_Y(\Omega)$.
\\
We stress that $L_{\overline{x}}$ belongs to a class of operators for which many results are well known in literature, namely $L_{\overline{x}}$ is a \emph{sublaplacian} on the Carnot group $\mathbb{G}$ (modulo a change of variables): indeed, denoting by $B=[b_{ij}(\overline{x})]_{i,j}$ the unique symmetric positive-definite matrix such that $B^2=A(\overline{x})$ and setting
    \begin{equation*}
        W_i :=  b^{ij}(\overline{x}) \, Y_j \hspace{5mm} 1 \leq i \leq m,
    \end{equation*}
     we get $L_{\overline{x}} = \sum_{i=1}^m W_i^2$, being $W_1, \, ..., \, W_m$ Lie-generators of $\mathbb{G}$ by construction. Then the following deep result holds true for $L_{\overline{x}}$:

\vspace{2mm}

\begin{theorem}[Fundamental solution of $L_{\overline{x}}$]
    \label{freezed fund sol}
    $L_{\overline{x}}$ admits a \emph{fundamental solution} $\Gamma_{\overline{x}}$, namely a function $\Gamma_{\overline{x}}:\mathbb{R}^{N} \setminus \{0\} \to \mathbb{R}$ that admits continuous derivatives up to order $2$ along $\{Y_i\}_{i=1}^m$ (in its domain of definition), solves the equation $L_{\overline{x}}u=0$ outside the origin and satisfy
    \begin{equation}
    \label{rep formula u: L Gamma = delta}
        u(x) = \int_{\mathbb{R}^N} \Gamma_{\overline{x}}(y^{-1} * x) \, L_{\overline{x}}u(y) \, dy \quad \forall u \in C^{\infty}_0(\mathbb{R}^N).
    \end{equation}
 Moreover, $\Gamma_{\overline{x}}$ is smooth outside the origin, $(2-Q)$-homogeneous and satisfies the following uniform estimates:
 \begin{equation*}
     \Theta_{\Gamma_{\overline{x}},p} := \sup_{|I| \leq p} \, \sup_{||w||=1} |Y_I \Gamma_{\overline{x}}(w)| \leq c(\mathbb{G}, \Lambda, p) \hspace{5mm} \forall p \geq 0.
 \end{equation*}
\end{theorem}

\vspace{2mm}

\hspace{-7mm} For the proof see \cite[Theorem 6.18]{librone}. Actually, we remark that the representation formula (\ref{rep formula u: L Gamma = delta}) holds true more in general for any $u \in C^{2,\alpha}_{0,Y}(\mathbb{R}^N)$, as can be shown via an approximation argument in Carnot groups (see \cite[Proposition 8.49]{librone}).

\vspace{2mm}

\hspace{-7mm} We point out that the reason why $L_{\overline{x}}$ plays an important role in our procedure is that, in view of the Hölder-continuity of the coefficients $a_{ij}$, the operator $L_{\overline{x}}$ is a \say{good approximation} (at least locally, near $\overline{x}$) of our original operator $L$, in the sense that the approximation error $L_{\overline{x}}-L$ is \say{small}, in a suitable sense. What we mean by that is made precise by the statement of the following lemma, which gives a quantitative control on the action of $L_{\overline{x}}-L$:

\vspace{2mm}

\begin{lemma}
\label{C^alpha estimate for L_x-L}
    For every $R>0$, $\overline{x} \in \mathbb{R}^N$ and $u \in C^{2,\alpha}_{0,Y}(B_R(\overline{x}))$ it holds
    \begin{align*}
        & |(L_{\overline{x}}-L)u|_{C^{\alpha}_Y(B_R(\overline{x}))} \leq c(
        \alpha, ||a||_0) \, R^{\alpha} |D^2 u|_{C^{\alpha}_Y(B_R(\overline{x}))}.
    \end{align*}
    Further requiring $a_{ij} \in C^{k,\alpha}_Y(B_R(\overline{x}))$, $k \geq 1$ for any $1 \leq i, \, j \leq m$, then for every $u \in C^{k+2,\alpha}_{0,Y}(B_R(\overline{x}))$ and $|I|=k$ we have
\begin{align*}
    & |Y_I (L_{\overline{x}} - L) u|_{C^{\alpha}_Y(B_R(\overline{x}))} \leq c(\alpha, ||a||_k) \left\{ R^{\alpha} |D^{k+2} u|_{C^{\alpha}_Y(B_R(\overline{x}))} + ||u||_{C^{k+1}_Y(B_R(\overline{x}))} \right\}.
\end{align*}
\end{lemma}

\vspace{2mm}

\begin{proof}
    Recalling the definitions (\ref{our operator L_Y Carnot}),(\ref{our freezed operator Carnot}), for any $x,y \in B_R(\overline{x})$ one finds    \begin{align*}
        & (L_{\overline{x}}-L)u(x)-(L_{\overline{x}}-L) u(y) = [a^{hl}(\overline{x})-a^{hl}(x)] \times \\
        & \times [Y_h Y_l u(x)- Y_h Y_l u(y)]+ [a^{hl}(y)-a^{hl}(x)] \, Y_h Y_l u(x') \quad \forall x, \, y \in B_R(\overline{x}).
    \end{align*}
    By Hölder continuity of $\{a_{hl}\}_{h,l}$ and equivalence between $d_Y$ and $d$
    \begin{equation*}
        |a^{hl}(x_1)-a^{hl}(x_2)| \leq c(||a||_0) \, d(x_1,x_2)^{\alpha} \hspace{5mm} \forall x_1,\, x_2 \in B_R(\overline{x}), \, 1 \leq h,l \leq m
    \end{equation*}
    Moreover, since $Y_h Y_l u \in C^{\alpha}_{0,Y}(B_R(\overline{x}))$, by (\ref{Hölder norm vs seminorm for u that vanishes at a point}) we have
    \begin{align}
    & \label{Hölder estimate Y_h Y_l u}
    |Y_h Y_l u(y)| \leq ||Y_h Y_l u||_{L^{\infty}(B_R(\overline{x}))} \leq c(\alpha) R^{\alpha} |Y_h Y_l u|_{C^{\alpha}_Y(B_R(\overline{x}))} 
\end{align}
    hence
    \begin{align*}
        & |(L_{\overline{x}}-L)u(x)-(L_{\overline{x}}-L) u(y)| \leq c(\alpha, ||a||_0) d(x,y)^{\alpha} R^{\alpha} |D^2 u|_{C^{\alpha}_Y(B_R(\overline{x}))}
    \end{align*}
    and the first claim follows just by definition of the Hölder seminorm $|\cdot|_{C^{\alpha}_Y}$.

    Let us turn to the second part.
    Fixed an integer $k \geq 1$ and a multindex $I$ with $|I|=k$, by iterated Leibniz rule we get
    \begin{align*}
        &Y_I(L_{\overline{x}}-L)u(x) = \sum_{J \subseteq I} \, Y_{I \setminus J} \left(a^{hl}(\overline{x}) - a^{hl}(x)\right) \, Y_J Y_h Y_l u(x),
    \end{align*}
    where both $J$ and $I \setminus J$ are meant as ordered submultindex of $I$, and also the case $J = \emptyset$ is considered. Then
    \begin{align}
        & Y_I(L_{\overline{x}}-L)u(x) - Y_I(L_{\overline{x}}-L)u(y) = \sum_{J \subseteq I} \{ Y_{I \setminus J} (a^{hl}(\overline{x})-a^{hl}(x)) \times \nonumber \\
        & \times \left[Y_J Y_h Y_l u(x) - Y_J Y_h Y_l u(y) \right] + (Y_{I \setminus J} a^{hl}(y)-Y_{I \setminus J} a^{hl}(x)) \, Y_J Y_h Y_l u(y) \}\nonumber\\
        & \label{rep formula Y_I (L_x-L)} = \sum_{J \subseteq I} \{ A_J^{(1)}(x,y)+A_J^{(2)}(x,y)\}.
    \end{align}
    To shorten the notation, let us denote $B_R(\overline{x})$ just by $B_R$. As for $A_J^{(1)}$:
    \begin{align*}
        & |A_{I}^{(1)}(x,y)| \leq d(x,y)^{\alpha} |a^{hl}|_{C^{\alpha}_Y(B_R)} d(x,\overline{x})^{\alpha} |Y_I Y_h Y_l u|_{C^{\alpha}_Y(B_R)} \\
        & \quad \leq c(||a||_0) \, R^{\alpha} \, d(x,y)^{\alpha} |D^{k+2} u|_{C^{\alpha}_Y(B_R)};\\
        & |A_J^{(1)}(x,y)| \leq ||Y_{I \setminus J} a^{hl}||_{L^{\infty}(B_R)} \, d(x,y)^{\alpha} |Y_J Y_h Y_l u|_{C^{\alpha}_Y(B_R)} \\
        & \quad \leq c(||a||_k) \, d(x,y)^{\alpha} ||u||_{C^{k+1,\alpha}_Y(B_R)} \quad \forall J \subsetneq I.
    \end{align*}
    Similarly, with also the aid of (\ref{Hölder norm vs seminorm for u that vanishes at a point}), for any $J \subseteq I$ we obtain
    \begin{align*}
        & |A_{J}^{(2)}(x,y)| \leq |Y_{I \setminus J} a^{hl}|_{C^{\alpha}_Y(B_R)} \, d(x,y)^{\alpha} ||Y_{J} Y_h Y_l u||_{L^{\infty}(B_R)} \leq |Y_{I \setminus J} a^{hl}|_{C^{\alpha}_Y(B_R)} \times \\
        & \times d(x,y)^{\alpha} |Y_{J} Y_h Y_l u|_{C^{\alpha}(B_R)} \leq c(\alpha, ||a||_k) \, R^{\alpha} \, d(x,y)^{\alpha} |D^{k+2} u|_{C^{\alpha}_Y(B_R)}.
    \end{align*}
    Plugging the above estimates in (\ref{rep formula Y_I (L_x-L)}) one gets
    \begin{align*}
        & |Y_I (L_{\overline{x}}-L) u|_{C^{\alpha}_Y(B_R)} \leq c(\alpha, ||a||_k) \left\{ R^{\alpha} |D^{k+2} u|_{C^{\alpha}_Y(B_R)} + ||u||_{C^{k+1,\alpha}_Y(B_R)}\right\}.
    \end{align*}
\end{proof}

\vspace{2mm}

\hspace{-7mm} We will also exploit the following abstract result, that holds for any $(2-Q)$-homogeneous function that is smooth outside the origin. Note that, by Theorem \ref{freezed fund sol}, the fundamental solution $\Gamma_{\overline{x}}$ of the operator $L_{\overline{x}}$ has these properties, hence everything in the following theorem holds true in particular for $\Gamma_{\overline{x}}$.

\vspace{2mm}

\begin{theorem}
\label{librone 6.32-6.33-8.24}
    Let $f\in C^{\infty}(\mathbb{R}^N \setminus \{0\})$ be homogeneous of degree $2-Q$, and 
    \begin{equation}
        \label{uniform estimates Y_I Gamma} \Theta_{f,k}:=\sup_{|I| \leq k} \, \sup_{||w||=1} |Y_I f(w)|.
    \end{equation}
    for any $k \geq 0$. Then, for any $1 \leq h, \, l \leq m$, there holds
        \begin{equation}
        \label{vanishing prop Y_h Y_l f}
            \int_{r<||y||<R} Y_h Y_l f(y)\, dy = 0 \quad \forall R>r>0,
        \end{equation}
        and for all $x \in \mathbb{R}^N$ and $\phi \in C^{2,\alpha}_{0,Y}(\mathbb{R}^N)$ we have
        \begin{equation}
        \label{rep formula first derivatives (DCT)}
            Y_l \int_{\mathbb{R}^N} f(y^{-1}*x) \, \phi(y) \, dy = \int_{\mathbb{R}^N} Y_l f(y^{-1}*x) \, \phi(y) \, dy;
        \end{equation}
        \begin{align}
        & \nonumber
            Y_h \int_{\mathbb{R}^N} Y_l f(y^{-1}*x) \, \phi(y) \, dy = P.V. \int_{\mathbb{R}^N} Y_h Y_l f(y^{-1}*x) \, \phi(y) \, dy + c_{hl} \phi(x)\\
        & \label{P.V. formula second derivatives}:= \lim_{\epsilon \to 0^+} \int_{d(x,y)>\epsilon} Y_h Y_l f(y^{-1}*x) \, \phi(y) \, dy + c_{hl} \phi(x)
        \end{align}
        with
        \begin{equation}
            \label{explicit expression c_hl} c_{hl}=\int_{||w||=1} Y_l f(w) \left( Y_h(w) \cdot \nu(w) \right) \, d\sigma(w),
        \end{equation}
        being $\nu$ the outer normal vector to the surface $\{||w||=1\}$. Moreover, for any $R_0>0$, there exist constants $c=(\mathbb{G}, \alpha, \Theta_{f,3},R_0)>0$ and $M=M(\mathbb{G})>1$ s.t.
    \begin{equation}
    \label{elliptic estimate Y_h Y_l f}
            |Y_h Y_l f(x)| \leq c \, ||x||^{-Q} \quad \forall ||x|| \leq R_0,
        \end{equation}
        and for every $x, \, y \in \mathbb{R}^N$ satisfying $||y*x|| \, \leq R_0$ and $||x|| \geq M||y||$ we have
        \begin{equation}
        \label{mean value ineq 6.25}
            |Y_h Y_l f(y * x)-Y_h Y_l f(x)| \leq c \, ||y|| \cdot ||x||^{-Q-1}.
        \end{equation}
\end{theorem}

\vspace{2mm}

\hspace{-7mm} For the proof of (\ref{vanishing prop Y_h Y_l f}) see \cite[Theorem 6.31]{librone}; (\ref{rep formula first derivatives (DCT)}) is straightforward from (\ref{rep formula u: L Gamma = delta}) by dominated convergence, since $y \mapsto Y_jf(y^{-1}*x) \phi(y) \in L^1_{loc}(\mathbb{R}^N)$. With (\ref{rep formula first derivatives (DCT)}) at hand, see \cite[Theorem 6.33]{librone} for (\ref{P.V. formula second derivatives})-(\ref{explicit expression c_hl}) under the stronger assumption that $\phi \in C^{\infty}_{0}(\mathbb{R}^N)$), and \cite[Proposition 8.49]{librone} for the extension to $\phi \in C^{2,\alpha}_{0,Y}(\mathbb{R}^N)$; see \cite[Lemma 8.24]{librone} for (\ref{elliptic estimate Y_h Y_l f}) and (\ref{mean value ineq 6.25}).

\vspace{2mm}

We conclude this section with some useful comments on the above theorem. First, the vanishing property (\ref{vanishing prop Y_h Y_l f}) actually makes the integral of $Y_i Y_j f$ (with $f$ as in Theorem \ref{librone 6.32-6.33-8.24}) vanish not only on spherical shells, but on several type of radially symmetric domains, indeed:
    \begin{equation}
        P.V. \int_{||x||<r} Y_i Y_j f(x) \, dx = \lim_{\epsilon \to 0^+} \int_{\epsilon < ||x|| < r} Y_i Y_j f(x) \, dx = 0 \hspace{5mm} \forall r>0.
    \end{equation}
    Moreover, from dominated convergence (based on (\ref{elliptic estimate Y_h Y_l f})), there holds
    \begin{align*}
        & \int_{||\eta||=1} (Y_i Y_j f)(\eta) \, dS(\eta) = \lim_{\epsilon \to 0^+} \int_{1-\epsilon < ||x|| < 1+\epsilon} (Y_i Y_j f)(x) \, dx = 0
    \end{align*}
    and consequently, for any radial smooth function $\psi(x) = \Psi(||x||)$ with compact support, by $(-Q)$-homogeneity of $Y_i Y_j f$ we have
    \begin{align}
        & \nonumber \int_{||w||>\epsilon} (\psi Y_i Y_j f)(w) \, dw = \int_{||w||>\epsilon} \Psi(||w||) \, Y_i Y_j f(D_{||w||^{-1}}(w)) \, ||w||^{-Q} \, dw \\
        & \label{vanishing int with radial psi} = \int_{\epsilon}^{+\infty}\Psi(\rho) \, \rho^{-Q} \int_{||\eta||=1} Y_i Y_j f(\eta) \, dS(\eta) = 0 \quad \forall \epsilon > 0,
    \end{align}
    hence by taking the limit as $\epsilon \to 0^+$:
    \begin{align}
    \label{vanishing P.V. int on R^N}
        P.V.  \int_{\mathbb{R}^N} (\psi Y_i Y_j f)(w) \, dw = 0.
    \end{align}
    
\vspace{2mm}

Furthermore, in force of (\ref{explicit expression c_hl}) and (\ref{elliptic estimate Y_h Y_l f}), we have the following control on $c_{hl}$ (here below we set $Y_i=\sum_{k=1}^N b_{ik}(x) \partial_{x_k}$):
    \begin{align*}
        & |c_{hl}| \leq \Theta_{f,1} \int_{||w||=1}\sum_{i=1}^N |b_i(w)| \, dw \leq \Theta_{f,1} \, \omega_Q \sum_{k=1}^N ||b_i||_{L^{\infty}(\partial B_1(0))} = c(\mathbb{G}) \, \Theta_{f,1},
    \end{align*}
    where $\omega_Q$ denotes the measure of the unit ball $B_1(0)$.

\vspace{5mm}

\section{Global Schauder estimates in Carnot groups} \label{sec 3}
\subsection{Local estimates for second order derivatives}

The first step towards the proof of Theorem \ref{1.4: Schauder est Carnot thm} consists in showing a similar estimate in local form, namely just for functions $u \in C^{2,\alpha}_Y(\mathbb{R}^N)$ with \say{small} compact support. Specifically, we begin by proving the following:

\vspace{2mm}

\begin{theorem}[Local Schauder estimates in Carnot groups]
\label{local C^alpha estimate thm}
    Under the assumptions of Theorem \ref{1.4: Schauder est Carnot thm}, there exist positive constants $c_0$ and $R_0$, depending only on $\mathbb{G}, \, \Lambda, \, \alpha, \, ||a||_0$, such that for every $0<R \leq R_0$ and for every $u \in C^{2, \alpha}_Y(\mathbb{R}^N)$ with $supp(u) \subseteq B_R(\overline{x})$ the following inequality holds:
    \begin{equation}
    \label{||Y_h Y_l u|| <= C ||Lu||}
        ||D^2 u||_{C^{\alpha}_Y(B_R(\overline{x}))} \leq c_0 \, |Lu|_{C^{\alpha}_Y(B_R(\overline{x}))}.
    \end{equation}
\end{theorem}

\vspace{2mm}

\hspace{-7mm} For a reason that will become clear later, we prove the following general form of the above theorem:

\vspace{2mm}

\begin{theorem}
\label{C^alpha estimate for T_f phi, 3.2}
    Let $f \in C^{\infty}(\mathbb{R}^N \setminus \{0\})$ be a $(2-Q)$-homogeneous function, let $1 \leq h,\, l \leq m$ and $T_f$ be the operator defined as
    \begin{equation*}
        (T_f \phi)(x):=Y_h Y_l \int_{\mathbb{R}^N} f(y^{-1} *x) \, \phi(y) \, dy \hspace{5mm} \forall \phi \in C^{\alpha}_{0,Y}(\mathbb{R}^N), \, x \in \mathbb{R}^N.
    \end{equation*}
    Then there exists a constant $c_0=c_0(\mathbb{G}, \alpha, \Theta_{f,3})>0$ such that for every $0<R \leq 1$, $\overline{x} \in \mathbb{R}^N$ and $\phi \in C^{\alpha}_Y(\mathbb{R}^N)$ with $supp(\phi) \subset B_R(\overline{x})$ it holds
    \begin{equation*}
        ||T_f \phi||_{C^{\alpha}_Y(B_R(\overline{x}))} \leq c_0 \, |\phi|_{C^{\alpha}_Y(B_R(\overline{x}))}.
    \end{equation*}
\end{theorem}

\vspace{2mm}

\hspace{-7mm} For the proof, we will exploit multiple times the following result:

\vspace{2mm}

\begin{lemma}
\label{spherical integral lemma}
    For any $\beta>0$, there exists a constant $c=c(\beta)>0$ such that for all $r_2>r_1>0$ there hold
    \begin{align*}
        & \int_{d(x,y)>r_2} d(x,y)^{-\beta-Q} \, dy \leq c r_2^{-\beta}; \\
        & \int_{r_1<d(x,y)<r_2} d(x,y)^{-Q} \, dy \leq  c \, \log\left( \frac{r_2}{r_1} \right); \\
        & \int_{d(x,y)<r_1} d(x,y)^{\beta-Q} \, dy \leq c r_1^{\beta}.
    \end{align*}
\end{lemma}

\vspace{2mm}

For the proof see \cite[Corollary 1.6]{Fo2}.

\vspace{2mm}

\begin{proof}[Proof of Theorem \ref{C^alpha estimate for T_f phi, 3.2}]
    The proof follows the lines of \cite[Lemma 2.7]{Hölder spaces: original vs lifted}, which proves a similar result for more general integral operators whose kernel is assumed to be $(-Q)$-homogeneous, smooth and to satisfy suitable estimates, similar to (\ref{elliptic estimate Y_h Y_l f})-(\ref{mean value ineq 6.25}). \\Let $0<R \leq 1$ and $\overline{x} \in \mathbb{R}^N$. Let us denote by $\kappa \geq 1$ the constant of the quasi-triangle inequality of $d(\cdot,\cdot)$; for all $z, \, \zeta \in B_R(\overline{x})$ we have
    \begin{equation*}
        ||z^{-1} * \zeta|| \leq \kappa(||z^{-1}||+||\zeta||) < \kappa(R+R) \leq 2 \kappa,
    \end{equation*}
     (this is crucial to apply properties (\ref{elliptic estimate Y_h Y_l f}) and (\ref{mean value ineq 6.25})). Given an arbitrary $x \in B_R(\overline{x})$, by (\ref{rep formula first derivatives (DCT)}) and (\ref{P.V. formula second derivatives}) we immediately get
    \begin{align*}
        & (T_f \phi)(x) = \lim_{\epsilon \to 0^+} \int_{d(x,y)>\epsilon} Y_h Y_l f(y^{-1} * x) \, \phi(y) \, dy + c_{hl} \phi(x)
    \end{align*}
    and since $\phi$ vanishes outside $B_R(\overline{x}) \subset B_{2\kappa R}(x) \subset B_{2\kappa}(x)$, we have
    \begin{align}
        & \nonumber (T_f \phi)(x) = \lim_{\epsilon \to 0^+} \int_{\epsilon < d(x,y) < 2 \kappa} Y_h Y_l f(y^{-1} * x) \, \phi(y) \, dy + c_{hl} \phi(x) \\
        & \nonumber \text{by (\ref{vanishing prop Y_h Y_l f})} \\
        & \nonumber = \lim_{\epsilon \to 0^+} \left\{\int_{\epsilon < d(x,y) < 2\kappa} Y_h Y_l f(y^{-1} * x) [\phi(y)-\phi(x)] \, dy \right\} + c_{hl} \phi(x) \\
        & \nonumber (\text{the above integral is absolutely convergent}) \\
        & \label{split of T_f phi in I and II} = \int_{d(x,y) < 2\kappa} Y_h Y_l f(y^{-1} * x) [\phi(y)-\phi(x)] \, dy + c_{hl} \phi(x) = I(x)+II(x).
    \end{align}
    We estimate the Hölder norms of $I$ and $II$ separately. As for $II$, 
    recalling that $|c_{hl}| \leq c(\mathbb{G}) \, \Theta_{f,1}$ and that $\phi(y)=0$ if $y \notin B_R(\overline{x})$, by (\ref{Hölder norm vs seminorm for u that vanishes at a point}) we readily get
    \begin{align}
        & \label{Hölder estimate for II} ||II||_{C^{\alpha}_Y(B_R(\overline{x}))} \leq c(\mathbb{G}) \, \Theta_{f,1} \, ||\phi||_{C^{\alpha}_Y(B_R(\overline{x}))}\leq c'(\mathbb{G},\alpha) \, \Theta_{f,1} \, |\phi|_{C^{\alpha}_Y(B_R(\overline{x}))}.
    \end{align}
    We turn to $I$: by (\ref{elliptic estimate Y_h Y_l f}) and Lemma \ref{spherical integral lemma} we have
    \begin{align*}
        & |I(x)| \leq c(\mathbb{G},  \Theta_{f,3}) |\phi|_{C^{\alpha}_Y(B_R(\overline{x}))} \int_{d(x,y)<2\kappa} d(x,y)^{\alpha-Q} \, dy\\
        & \leq c(\mathbb{G}, \Theta_{f,3}, \alpha) |\phi|_{C^{\alpha}_Y(B_R(\overline{x}))}
    \end{align*}
    for all $x \in B_R(\overline{x})$, so
    \begin{equation}
        \label{L^inf estimate for I}
        ||I||_{L^{\infty}(B_R(\overline{x}))} \leq c(\mathbb{G}, \Theta_{f,3}, \alpha) |\phi|_{C^{\alpha}_Y(B_R(\overline{x}))}.
    \end{equation} 
    As for the Hölder seminorm of $I$, let $\theta \in C^{\infty}_{0}(\mathbb{R}^N; [0,1])$ be a radial function such that $\theta \equiv 1$ in $B_{2\kappa}(0)$ and $\theta \equiv 0$ outside $B_{4\kappa}(0)$;
    we can write $I$ as
    \begin{align*}
        & I(x) = P.V. \int_{\mathbb{R}^N} (\theta Y_i Y_j f)(y^{-1}*x) \, \phi(y) \, dy= P.V. \int_{\mathbb{R}^N} (\theta Y_i Y_j f)(y^{-1}*x) \times\\
        & \times [\phi(y)-\phi(x)] \, dy + \phi(x) \cdot P.V. \int_{\mathbb{R}^N} (\theta Y_i Y_j f)(y^{-1}*x) \, dy\\
        & \text{by (\ref{vanishing P.V. int on R^N}) and absolute convergence of the first integral} \\
        & = \int_{\mathbb{R}^N} (\theta Y_i Y_j f)(y^{-1}*x) [\phi(y)-\phi(x)] \, dy + c_{hl} \phi(x)
        \quad \forall x \in B_R(\overline{x}).
    \end{align*}
    Then, given any $x, \, x' \in B_R(\overline{x})$, we have
    \begin{align*}
        & I(x)-I(x') = \int_{\mathbb{R}^N} \{(\theta Y_i Y_j f)(y^{-1}*x) [\phi(y)-\phi(x)] \, dy\\
        & -(\theta Y_i Y_j f)(y^{-1}*x')[\phi(y)-\phi(x')]\} \, dy \\
        & = \left(\int_{d(x,y) \geq Md(x,x')} + \int_{d(x,y)<Md(x,x')} \right) \{...\} \, dy = A(x,x')+B(x,x') 
    \end{align*}
    and we estimates these two terms separately. As for $A$, we can write
    \begin{align*}
        & A(x,x')=\int_{d(x,y) \geq Md(x,x')} [(\theta Y_h Y_l f)(y^{-1}*x)-(\theta Y_h Y_l f)(y^{-1}*x')] \times \\
        & \times [\phi(y)-\phi(x')] \, dy + [\phi(x')-\phi(x)] \int_{d(x,y) \geq Md(x,x')} (\theta Y_h Y_l f)(y^{-1}*x) \, dy
        \\
        & \text{again by (\ref{vanishing P.V. int on R^N})}
        \\
        & = \int_{d(x,y) \geq Md(x,x')} [(\theta Y_h Y_l f)(y^{-1}*x)-(\theta Y_h Y_l f)(y^{-1}*x')][\phi(y)-\phi(x')]\, dy
    \end{align*}
    We further split the integration domain $\{d(x,y) \geq Md(x,x')\}$ in two parts depending on whether $d(x',y)>4\kappa$ or $d(x',y) \leq 4\kappa$; we then split $A$ as the sum of the correspondent two terms, denoted respectively $A_1$ and $A_2$. As for $A_1$ (i.e., when $d(x',y)>4\kappa$), by quasi-triangle inequality
    \begin{equation*}
        d(y,\overline{x})>\kappa^{-1}d(x',y)-d(x,x')>4-R>R,
    \end{equation*}
    hence both $(\theta Y_h Y_l f)(y^{-1}*x')$ and $\phi(y)$ vanishes; it follows that
    \begin{equation*}
        A_1(x,x') = -\phi(x') \int_{\{d(x,y) \geq Md(x,x')\} \cap \{d(x',y)>4\kappa\}} (\theta Y_h Y_l f)(y^{-1}*x) \, dy = 0.
    \end{equation*}
    by (\ref{vanishing int with radial psi}). As for $A_2$, by Lemma \ref{UGE Lemma 7.6} and the boundedness of $\{Y_i \theta\}_{i=1}^m$, we have
    \begin{align*}
        & |\theta(y^{-1}*x)-\theta(y^{-1}*x')| \leq \overline{c}d(x,x') \sup_{d(x',z) \leq \overline{c} d(x,x')} \sum_{i=1}^m |Y_i \theta(y^{-1} * z)| \leq \overline{\overline{c}} d(x,x').
    \end{align*}
    Then, in view also of (\ref{elliptic estimate Y_h Y_l f})-(\ref{mean value ineq 6.25}), for every $d(x,y) \geq Md(x,x')$ there holds
    \begin{align*}
        & |(\theta Y_h Y_l f)(y^{-1}*x)-(\theta Y_h Y_l f)(y^{-1}*x')| \leq |\theta(y^{-1}*x) \cdot [(Y_h Y_l f)(y^{-1}*x)\\
        & -(Y_h Y_l f)(y^{-1}*x')]|+ |[\theta(y^{-1}*x)-\theta(y^{-1}*x')] \cdot (Y_h Y_l f)(y^{-1}*x')| \\
        & \leq c_1(\Theta_{f,3}) \, d(x,x') \, d(x',y)^{-Q-1} + c_2(\Theta_{f,2}) \, d(x,x') \, d(x',y)^{-Q},
    \end{align*}
    and since we are considering $d(x',y)<4\kappa$
    \begin{equation*}
        |(\theta Y_h Y_l f)(y^{-1}*x)-(\theta Y_h Y_l f)(y^{-1}*x')| \leq c(\Theta_{f,3}) \, d(x,x') \, d(x',y)^{-Q-1}.
    \end{equation*}
    This in turn gives 
    \begin{align*}
        & |A_2(x,x')| \leq c(\Theta_{f,3}) |\phi|_{C^{\alpha}_Y(B_R(\overline{x}))} d(x,x') \int_{d(x,y) \geq Md(x,x')^{\alpha}} d(x',y)^{-1-Q+\alpha} dy \\
        & \leq c(\Theta_{f,3},\alpha)  |\phi|_{C^{\alpha}_Y(B_R(\overline{x}))} d(x,x')^{\alpha}.
    \end{align*}
    As for $B$, we further split
    \begin{align*}
        & |B(x,x')| \leq \int_{d(x,y)<Md(x,x')} |(\theta Y_i Y_j f)(y^{-1}*x) [\phi(y)-\phi(x)| \, dy \\
        & + \int_{d(x,y)<Md(x,x')} |(\theta Y_i Y_j f)(y^{-1}*x') [\phi(y)-\phi(x')]| \, dy = B_1(x,x')+B_2(x,x').
    \end{align*}
    By quasi-triangle inequality, for every $d(x,y)<Md(x,x')$ we have
    \begin{align*}
        d(x',y) \leq \kappa \, (d(x,y)+d(x,x')) \leq \kappa (M+1) d(x,x') =: M' d(x,x'),
    \end{align*}
    hence
    \begin{align*}
        & |B_1(x,x')| \leq c(\Theta_{f,3}) |\phi|_{C^{\alpha}_Y(B_R)} \int_{d(x,y)<Md(x,x')} d(x,y)^{\alpha-Q} \, dy \\
        & \hspace{5mm}\leq c(\Theta_{f,3}) |\phi|_{C^{\alpha}_Y(B_R)} d(x,x')^{\alpha};
        \\
        & |B_2(x,x')| \leq c'(\Theta_{f,3}) |\phi|_{C^{\alpha}_Y(B_R(\overline{x}))} \int_{d(x',y)<M'd(x,x')} d(x',y)^{\alpha-Q} dy \\
        & \quad \leq c'(\Theta_{f,3}) |\phi|_{C^{\alpha}_Y(B_R(\overline{x}))} d(x,x')^{\alpha}.
    \end{align*}
    Together with the estimate of $A_2$, the latter give
    \begin{align*}
        |I|_{C^{\alpha}_Y(B_R(\overline{x}))} \leq c(\Theta_{f,3}) |\phi|_{C^{\alpha}_Y(B_R(\overline{x}))},
    \end{align*}
    with $c$ independent of $R$. Plugging the above with (\ref{Hölder estimate for II})-(\ref{L^inf estimate for I}) in (\ref{split of T_f phi in I and II}) we achieve
    \begin{equation*}
        ||T_f \phi||_{C^{\alpha}_Y(B_R(\overline{x}))} \leq c(\Theta_{f,3}) |\phi|_{C^{\alpha}_Y(B_R(\overline{x}))}.
    \end{equation*}
\end{proof}

\vspace{2mm}

\hspace{-7mm} As a sort of special case of the above theorem, we deduce the:

\vspace{2mm}

\begin{proof}[Proof of Theorem \ref{local C^alpha estimate thm}]
    Let $0 < R \leq 1$ to be chosen later, $\overline{x} \in \mathbb{R}^N$ and $u$ as in the statement of the theorem. Since $L_{\overline{x}}u \in C^{\alpha}_{0,Y}(B_{R}(\overline{x}))$ and $\Gamma_{\overline{x}}$ is smooth outside the origin and $(2-Q)$-homogeneous by Theorem \ref{freezed fund sol}, then we can apply Theorem \ref{C^alpha estimate for T_f phi, 3.2} with $(f,\phi)=(\Gamma_{\overline{x}}, L_{\overline{x}}u)$: we obtain
    \begin{equation}
        \label{application of 3.2 to L_x(u) and Gamma_x} ||T_{\Gamma_{\overline{x}}}(L_{\overline{x}}u)||_{C^{\alpha}_Y(B_R(\overline{x}))} \leq c(\mathbb{G}, \alpha, \Theta_{\Gamma_{\overline{x}},3}) \, |L_{\overline{x}}u|_{C^{\alpha}_Y(B_R(\overline{x}))}
    \end{equation}
    where, recalling the definition of the operator $T_f$, by (\ref{rep formula u: L Gamma = delta}) we have
    \begin{align*}
        T_{\Gamma_{\overline{x}}}(L_{\overline{x}}u) = Y_h Y_l \int_{\mathbb{R}^N} \Gamma_{\overline{x}}(y^{-1}*x) \, L_{\overline{x}}u(y) \, dy = Y_h Y_l u(x).
    \end{align*}
    Moreover $\Theta_{\Gamma_{\overline{x}},3} \leq c(\mathbb{G}, \Lambda)$ by (\ref{uniform estimates Y_I Gamma}) and, in view of Lemma \ref{C^alpha estimate for L_x-L}:
    \begin{align*}
        & |L_{\overline{x}}u|_{C^{\alpha}_Y(B_R(\overline{x}))} \leq |Lu|_{C^{\alpha}_Y(B_R(\overline{x}))} + |(L_{\overline{x}}-L)u|_{C^{\alpha}_Y(B_R(\overline{x}))} \\
        & \quad \leq |Lu|_{C^{\alpha}_Y(B_R(\overline{x}))} + c(\alpha, ||a||_0) R^{\alpha} |D^2 u|_{C^{\alpha}_Y(B_R(\overline{x}))}.
    \end{align*}
    Plugging all these in the inequality (\ref{application of 3.2 to L_x(u) and Gamma_x}) and summing over $h, \, l=1, \, ..., \, m$:
    \begin{align*}
        & ||D^2  u||_{C^{\alpha}_Y(B_R(\overline{x}))} \leq c_1(\mathbb{G}, \Lambda, \alpha) |Lu|_{C^{\alpha}_Y(B_R(\overline{x}))}  \\
        & \hspace{2.55cm} + c_2(\mathbb{G}, \Lambda, \alpha, ||a||_0) R^{\alpha} |D^2u|_{C^{\alpha}_Y(B_R(\overline{x}))}.
    \end{align*}
    Setting $R_0=R_0(\mathbb{G}, \Lambda, \alpha, ||a||_0):=\min\left\{1, \, (2c_2)^{-1/\alpha} \right\}$, we deduce that for every $0 < R \leq R_0$ it holds
    \begin{equation*}
        ||D^2 u||_{C^{\alpha}_Y(B_R(\overline{x}))} \leq c(\mathbb{G}, \Lambda, \alpha) |Lu|_{C^{\alpha}_Y(B_R(\overline{x}))}.
    \end{equation*}
\end{proof}

\vspace{5mm}

\subsection{From local to global}
In order to achieve the global estimate (\ref{Schauder est Carnot eqn}) starting from the local version (\ref{||Y_h Y_l u|| <= C ||Lu||}), we need some interpolation inequalities to help us handling the Hölder norms of $u$ and its first order derivatives (note indeed that Theorem \ref{local C^alpha estimate thm} gives an estimate only for second order derivatives). We will exploit the following:

\vspace{2mm}

\begin{theorem}[Interpolation inequality]
    \label{Real interp ineq thm}
    Let $R>0$ and $\overline{x} \in \mathbb{R}^N$. There exist constants $c(\mathbb{G}, \Lambda)>0$ and $\gamma=\gamma(\mathbb{G})>1$ (both independent of $R$ and $\overline{x}$) such that for any $u \in C^{2,\alpha}_Y(B_R(\overline{x}))$, $0 < \epsilon < \frac{1}{3}$ and $\frac{R}{2}<r<R$ it holds
    \begin{equation*}
        ||u||_{C^{1,\alpha}_Y(B_r(\overline{x}))} \leq \epsilon \, \sum_{h=1}^m ||Y_h^2 u||_{L^{\infty}(B_R(\overline{x}))} + \frac{c}{\epsilon^\gamma (R-r)^{\gamma}} ||u||_{L^{\infty}(B_R(\overline{x}))}.
    \end{equation*}
\end{theorem}

\vspace{2mm}

\hspace{-7mm} For the proof see \cite[Proposition 8.56]{librone}.

With Theorems \ref{local C^alpha estimate thm}-\ref{Real interp ineq thm} at hand, we finally address the:

\vspace{2mm}

\begin{proof}[Proof of Theorem \ref{1.4: Schauder est Carnot thm}]
    Let $R>0$ such that the local estimate (\ref{||Y_h Y_l u|| <= C ||Lu||}) holds for every $v \in C^{2,\alpha}_Y(\mathbb{R}^N)$ vanishing outside a ball of radius $2R$. Let $\phi \in C^{\infty}_0(\mathbb{R}^N)$ s.t.
    \begin{equation*}
        \phi \equiv 1 \: \text{in} \, B_R(\overline{x}), \quad \phi \equiv 0 \: \text{outside} \, B_{2R}(\overline{x}), \quad 0 \leq \phi \leq 1 \: \text{in} \, \mathbb{R}^N.
    \end{equation*}
    Let $\{B_R(\overline{x_i})\}_{i \in \mathbb{N}}$ be a covering of $\mathbb{R}^N$ and $\varphi_i(x):=\phi(\overline{x_i}^{-1} * x)$ for every $i \geq 1$. \\Given $u \in C^{2,\alpha}_Y(\mathbb{R}^N)$, we have $u \varphi_i \in C^{2,\alpha}_Y(\mathbb{R}^N)$, $supp(u \varphi_i) \subset B_{2R}(\overline{x_i})$ and
    \begin{equation*}
        Y_h Y_l(u \varphi_i)= \varphi_i \cdot Y_h Y_l u + Y_h \varphi_i \cdot Y_l u + Y_l \varphi_i \cdot Y_l u + Y_h Y_l \varphi_i \cdot u \equiv Y_h Y_l u \hspace{5mm} \text{in } B_R(\overline{x_i}).
    \end{equation*}
    In view of the estimate (\ref{||Y_h Y_l u|| <= C ||Lu||}) for $u \varphi_i$ we have (for some $c_0=c_0(\mathbb{G}, \Lambda, \alpha, ||a||_0)$)
    \begin{equation}
    \label{intermediate est Hölder}
        ||Y_h Y_l u||_{C^{\alpha}_Y(B_R(\, \overline{x_i}))} \leq  ||Y_h Y_l (u \varphi_i)||_{C^{\alpha}_Y(B_{2R}(\overline{x_i}))} \leq c_0 \, ||L(u \varphi_i)||_{C^{\alpha}_Y(B_{2R}(\overline{x_i}))}.
    \end{equation}
    For simplicity, we denote $B_i:=B_{2R}(\overline{x_i})$ and $B_0:=B_{2R}(0)$. Note that 
    \begin{align*}
        & L(u \varphi_i) = \varphi_i \cdot Lu + u \cdot L\varphi_i + 2 a^{hl} \, Y_h u \cdot Y_l \varphi_i
    \end{align*}
    and exploiting the left-invariance of $Y_1,...,Y_m$
    \begin{align*}
        & ||\varphi_i||_{C^{\alpha}_Y(B_i)}=||\phi||_{C^{\alpha}_Y(B_0)}; \quad ||\varphi_i||_{C^{k,\alpha}_Y(B_i)} = ||\phi||_{C^{k,\alpha}_Y(B_0)} \quad \forall k \geq 1; \\
        & ||L \varphi_i||_{C^{\alpha}_Y(B_i)} \leq ||a^{hl}||_{C^{\alpha}_Y(B_i)} ||Y_h Y_l \varphi_i||_{C^{\alpha}_Y(B_i)} \leq c(||a||_0) ||\phi||_{C^{2,\alpha}_Y(B_0)}
    \end{align*}
    for all $i \in \mathbb{N}$. Then
    \begin{align*}
        & ||L(u \varphi_i)||_{C^{\alpha}_Y(B_i)} \leq ||\phi||_{C^{2,\alpha}_Y(\mathbb{R}^N)} \{ ||Lu||_{C^{\alpha}_Y(B_i)} + c(||a||)_0  ||u||_{C^{1,\alpha}_Y(B_i)} \}.
    \end{align*}
    Inserting this in (\ref{intermediate est Hölder}) and summing over $1 \leq h, l \leq m$ we get
    \begin{align*}
        & ||D^2 u||_{C^{\alpha}_Y(B_R(\overline{x_i}))} \leq c'_0 \left\{||Lu||_{C^{\alpha}_Y(B_{2R}(\overline{x_i}))} + ||u||_{C^{1,\alpha}_Y(B_{2R}(\overline{x_i}))} \right\}
    \end{align*}
    Adding $||u||_{C^{1,\alpha}_Y(B_{R}(\overline{x_i}))}$ to both sides and using Theorem \ref{Real interp ineq thm} we find
    \begin{align*}
        & ||u||_{C^{2,\alpha}_Y(B_R(\overline{x_i}))} \leq c_0'' \left\{ ||Lu||_{C^{\alpha}_Y(B_{2R}(\overline{x_i}))} + ||u||_{C^{1,\alpha}_Y(B_{2R}(\overline{x_i}))} \right\} \\
        & \leq c_0'' \, \left\{ ||Lu||_{C^{\alpha}_Y(B_{2R}(\overline{x_i}))} + \epsilon \sum_{h=1}^m ||Y_h^2 u||_{L^{\infty}(B_{3R}(\overline{x_i}))} + \frac{c}{\epsilon^{\gamma}} ||u||_{L^{\infty}(B_{3R}(\overline{x_i}))} \right\} \\
        & \leq c_0'' \, \left\{ ||Lu||_{C^{\alpha}_Y(
        \mathbb{R}^N)} + \epsilon \sum_{h=1}^m ||Y_h^2 u||_{L^{\infty}(\mathbb{R}^N)} + \frac{c}{\epsilon^{\gamma}} ||u||_{L^{\infty}(\mathbb{R}^N)} \right\}
    \end{align*}
    for any $0 < \epsilon < \frac{1}{3}$, for some positive constants $c_0''(\mathbb{G}, \lambda, \alpha, ||a||_0)$ and $c(\mathbb{G}, \Lambda)$.
    Since the latter holds true for every $i \geq 1$ with $c_0''$ and $c$ independent from $i$ by construction, then taking the supremum over $i \geq 1$ we obtain
    \begin{align*}
        & ||u||_{C^{2,\alpha}_Y(\mathbb{R}^N)} \leq c_0'' \, \left\{ ||Lu||_{C^{\alpha}_Y(
        \mathbb{R}^N)} + \epsilon \sum_{h=1}^m ||Y_h^2 u||_{L^{\infty}(\mathbb{R}^N)} + \frac{c}{\epsilon^{\gamma}} ||u||_{L^{\infty}(\mathbb{R}^N)} \right\}
    \end{align*}
    for every $0 < \epsilon < \frac{1}{3}$. Finally, choosing $\epsilon=\overline{\epsilon}$ such that $c_0'' \overline{\epsilon} < \frac{1}{2}$ we conclude:
    \begin{align*}
        & \hspace{-1.85cm} ||u||_{C^{2,\alpha}_Y(\mathbb{R}^N)} \leq 2c_0' ||Lu||_{C^{\alpha}_Y(\mathbb{R}^N)} + \frac{2c'_0 c_2}{(2c'_0 c_1)^{\gamma}} ||u||_{L^{\infty}(\mathbb{R}^N)} \\
        & \leq \widehat{c}_1 ||Lu||_{C^{\alpha}_Y(\mathbb{R}^N)} + \widehat{c}_2 ||u||_{L^{\infty}(\mathbb{R}^N)}
    \end{align*}
    with $\widehat{c}_1, \, \widehat{c}_2$ depending only on $\mathbb{G}, \, \Lambda,\,  \alpha, \, ||a||_0$.
\end{proof}

\vspace{2mm}

\section{Estimates for higher order derivatives}
\label{sec 4}
    Let us now prove estimates in Hölder spaces for higher-order derivatives of any smooth function, namely (\ref{global est higher derivs thm}). To this aim, we ask stronger regularity on the coefficients $a_{ij}$, that is $a_{ij} \in C^{k,\alpha}_Y(\mathbb{R}^N)$ for every $1 \leq i, \, j \leq m$, so that $Lu \in C^{k,\alpha}_Y(\mathbb{R}^N)$ for any $u \in C^{k+2,\alpha}_Y(\mathbb{R}^N)$.
\\
We point out that, in contrast with the corresponding Euclidean case, due to the noncommutativity of $\mathbb{G}$ the latter result is not an easy consequence of Theorem \ref{1.4: Schauder est Carnot thm}. 
Instead, what we need for the proof are suitable representation formulas of $Y_I u$ in terms of $L_{\overline{x}} u$, for any multindex $I$. 

\vspace{2mm}

\begin{theorem}
\label{rep formula higher derivs with Gamma^J}
    For every $k \geq 1$, $u \in C^{k+2,\alpha}_{Y,0}(\mathbb{R}^N)$ and $I=(i_1,...,i_k)$, the following representation formula holds:
    \begin{equation}
    \label{rep formula Y_I u}
        Y_I u(x) = \int_{\mathbb{R}^N} \Gamma_{\overline{x}}^{I,j_1...j_k}(y^{-1} * x) \, Y_{j_1}...Y_{j_k} L_{\overline{x}} u(y) \, dy,
    \end{equation}
    for suitable functions $\{\Gamma_{\overline{x}}^{I; \, j_1...j_k}\}_{1 \leq j_1,...,j_k \leq m} \in C^{\infty}(\mathbb{R}^N \setminus \{0\})$,  homogeneous of degree $2-Q$ and satisfying the following uniform estimate for every $p \geq 0$:
    \begin{equation}
        \label{uniform estimates for Gamma^J} \Theta_{\Gamma_{\overline{x}}^{i_1...i_k; \, j_1...j_k}, \, p} := \sup_{|H| \leq p} \, \sup_{||x||=1} |Y_H  \Gamma_{\overline{x}}^{i_1...i_k; \, j_1...j_k}(x)| \leq c(p,k,\mathbb{G}, \Lambda).
    \end{equation}
\end{theorem}

\vspace{2mm}

\begin{proof}
    For (\ref{rep formula Y_I u}) we follow the lines of \cite[Proposition 8.32]{librone}. Recall that
    \begin{equation*}
    Y_j u(x)=\int_{\mathbb{R}^N} (Y_j \Gamma_{\overline{x}})(y^{-1} * x) \, L_{\overline{x}}u(y) \, dy \hspace{5mm} \forall x \in \mathbb{R}^N, \, u \in C^{2,\alpha}_{0,Y}(\mathbb{R}^N), \, 1 \leq j \leq m.
\end{equation*}
    Since the right-invariant vector fields $\{Y_1^R, \, ..., \, Y_N^R\}$ are homogeneous of degree $\sigma_1, \, ..., \, \sigma_N$ (respectively) and form a basis of $\mathbb{R}^N$ at any point, then we can write
    \begin{equation*}
        Y_j = \sum_{i=1}^N r_{ij}(\cdot) \, Y_i^R,
    \end{equation*}
    where $r_{ij}$ is a $(\sigma_i-\sigma_j)$-homogeneous polynomial, for any $1 \leq i, \, j \leq N$. Moreover, exploiting the stratified structure of $\mathfrak{a}:=Lie(\mathbb{G})$ we have
    \begin{equation*}
        Y_i^R = \theta^{h_1,...,h_{\sigma_i}} Y_{h_1}^R \, ... \, Y_{h_{\sigma_i}}^R \quad \forall m+1 \leq i \leq N,
    \end{equation*}
    for suitable real constants $\{\theta^{h_1,...,h_{\sigma_i}}\}_{1 \leq h_1,...,h_{\sigma_i} \leq m}$. Since the degree of homogeneity of $r_{ij}$ is less that $\sigma_i$, by Proposition \ref{prop: translations and dilations} $r_{ij}$ and $Y_i^R$ commute; then
    \begin{align*}
        & Y_j u(x) = \int_{\mathbb{R}^N} \left[Y_i^R(r^{ij} \Gamma_{\overline{x}})\right](y^{-1} * x) \, L_{\overline{x}} u(y) \, dy + \\
        & + \int_{\mathbb{R}^N} \sum_{i=m+1}^N \theta^{h_1,...,h_{\sigma_i}} \left[ Y_{h_1}^R \, ... \, Y_{h_{\sigma_i}}^R(r^{ij} \Gamma_{\overline{x}}) \right](y^{-1} * x) \, L_{\overline{x}} u(y) \, dy.
    \end{align*}
    Loading one derivative on $L_{\overline{x}} u$ by means of Proposition \ref{prop: translations and dilations} and setting
    \begin{equation*}
        \Gamma_{\overline{x}}^{j; \, l}:=r_{lj} \Gamma_{\overline{x}} + \sum_{i=m+1}^N \theta^{l, \, h_2, \, ..., \, h_{\sigma_i}} Y_{h_2}^R \, ... \, Y_{h_{\sigma_i}}^R(r_{ij} \Gamma_{\overline{x}}),
    \end{equation*}
    we obtain
    \begin{equation}
    \label{new rep formula Y_j u (IPP)}
        Y_j u(x)=\int_{\mathbb{R}^N} \Gamma_{\overline{x}}^{j; \, l}(y^{-1} * x) \, Y_l L_{\overline{x}} u(y) \, dy.
    \end{equation}
    Iterating this argument we get (\ref{rep formula Y_I u}), with $\Gamma_{\overline{x}}^{i_1,...,i_k;\, j_1,...,j_k}$, $k \geq 1$ defined recursively as follows (here below we set $\Gamma_{\overline{x}}^{\emptyset;\emptyset} \equiv \Gamma_{\overline{x}}$ for homogeneity of notation):
    \begin{align}
        & \label{recursive formula Gamma^I;J} \Gamma_{\overline{x}}^{I; J} = r_{i_1 j_1} \Gamma_{\overline{x}}^{I \setminus \{i_1\}; J \setminus \{j_1\}} + \sum_{h=m+1}^N \, \theta^{j_1, \, l_2, \, ..., \, l_{\sigma_h}} Y_{l_2}^R \, ... \, Y_{l_{\sigma_h}}^R(r_{i_1 h} \Gamma_{\overline{x}}^{I \setminus \{i_1\}; J \setminus \{j_1\}}).
    \end{align}
    Now we prove the properties of $\Gamma_{\overline{x}}^{I;J}$ by induction on $k =|I| \geq 0$, the case $k=0$ being given by Theorem \ref{freezed fund sol}. Fixed $k \geq 1$, assume the thesis holds true for $k-1$ indices, then take $i_1,...,i_k, \, j_1,...,j_k$. The smoothness of $\Gamma_{\overline{z}}^{I;J}$ outside the origin is straightforward by (\ref{recursive formula Gamma^I;J}); thanks also to the $1$-homogeneity of each among $Y_{l_2}^R,...,Y_{l_{\sigma_h}}^R$ and the homogeneity properties of $\{r_{ij}\}_{i,j}$, we easily see that $\Gamma^{I;J}_{\overline{x}}$ is homogeneous of same degree as $\Gamma^{I\setminus \{i_1\}; J \setminus \{j_1\}}_{\overline{x}}$, which is $2-Q$ by inductive assumption.    
    As for the uniform estimates (\ref{uniform estimates for Gamma^J}), note that $r_{i_1 j_1}$ is homogeneous of degree $\sigma_{i_1}-\sigma_{j_1}=0$ (since $1 \leq i_1,j_1 \leq m$), hence it is constant; then, given any integer $p\geq 0$ and multindex $H$ with $|H|=p$, we have
    \begin{align*}
        & Y_{H} \Gamma_{\overline{x}}^{I;J} = r_{i_1 j_1} Y_{H} \Gamma_{\overline{x}}^{I \setminus \{i_1\}; J \setminus \{j_1\}}\\
        & +\sum_{h=m+1}^N \, \theta^{j_1,l_2...,l_{\sigma_h}} Y_{H} Y_{l_2}^R ... Y_{l_{\sigma_h}}^R(r_{i_1 h} \Gamma_{\overline{x}}^{I \setminus \{i_1\}; J \setminus \{j_1\}}).
    \end{align*}
    Denoting by $[q_{ij}]_{1 \leq i,j \leq N}$ the inverse matrix of $[r_{ij}]_{1 \leq i,j \leq N}$, we can write
    \begin{equation*}
        Y_i^R = \sum_{j=1}^N q_{ij}(\cdot) \, Y_j = \sum_{j=1}^m q_{ij} Y_j + \sum_{j=m+1}^N  \beta^{t_1...t_{\sigma_j}} q_{ij}(\cdot)Y_{t_1}...Y_{t_{\sigma_j}}
    \end{equation*}
    with $q_{ij}$ smooth and $(\sigma_j-\sigma_i)$-homogeneous, and $\{\beta^{t_1,...,t_{\sigma_j}}\}_{1 \leq t_1,...,t_{\sigma_j} \leq m}$ real constants. Combining the last two equalities, by Leibniz rule one can write
    \begin{align*}
        Y_{H} \Gamma_{\overline{x}}^{I;J}(x,t)=\sum_{L} \psi_{H,L}(x) \, Y_{L} \Gamma_{\overline{x}}^{I \setminus \{i_1\}; J \setminus \{j_1\}}(x,t).
    \end{align*}
    with $L$ ranging in a suitable set of multindices and each $\psi_{H,L}$ being the product of some among the functions $\{q_{ij}\}_{i,j}, \, \{r_{ij}\}_{i,j}$ and their derivatives of any order along $Y_1,...,Y_m$; in particular, $\{\psi_{H,L}\}_L$ are smooth functions. By homogeneity arguments $\psi_{H,L}$ is homogeneous of degree $(|H|-|L|)$, which must be nonnegative by smoothness of $\psi_{H,L}$: thus we can assert that $L$ ranges in $\{|L| \leq |H|\}$.
    Again by this smoothness, the suprema of $\{|\psi_{H,L}|\}_{|L| \leq |H| = p}$ on the unit sphere are (uniformly) bounded by a constant depending only on the structure of $\mathbb{G}$, so
    \begin{align*}
        & \sup_{|H| \leq p} \sup_{||x||=1} |Y_H \Gamma_{\overline{x}}^{I; \, J}(x)| \leq c(p) \sup_{|H| \leq |L| = p} \, \sup_{||x||=1} |\psi_{H,L}(x)| \times \\
        & \quad \times \sup_{|L| \leq p} \, \sup_{||x||=1} |Y_L \Gamma_{\overline{x}}^{I \setminus \{i_1\}; J \setminus \{j_1\}}(x)| \leq c(p,k,\mathbb{G},\Lambda)
    \end{align*}
        (the last inequality being due to our inductive assumption).
\end{proof}

\vspace{2mm}

\hspace{-7mm} It follows that, by analogy with the proof of Theorem \ref{local C^alpha estimate thm}, we can exploit the abstract result given by Theorem \ref{C^alpha estimate for T_f phi, 3.2} to get a local control on $(k+2)$-th order derivatives of $u$ in terms of $k$-th derivatives of $Lu$, for whatever $k \geq 1$. 

\vspace{2mm}

\vspace{2mm}

\begin{theorem}[Local estimates of higher order]
\label{local est high order der thm}
    For every $k \geq 1$, there exist constants $c_k=c(k,\mathbb{G}, \Lambda, \alpha)>0$ and $\overline{R}_k=\overline{R}(k,\mathbb{G}, \Lambda, \alpha, ||a||_k)>0$ such that for every $0 < R \leq \overline{R}_k$, $\overline{x} \in \mathbb{R}^N$ and $u \in C^{k+2,\alpha}_{Y}(B_R(\overline{x}))$ there holds
    \begin{equation*}
        ||D^{k+2}u||_{C^{\alpha}_Y(B_R(\overline{x}))} \leq c_k \left\{|D^k( Lu)|_{C^{\alpha}_Y(B_R(\overline{x}))} + ||u||_{C^{k+1,\alpha}_Y(B_R(\overline{x}))}\right\}.
    \end{equation*}
\end{theorem}

\vspace{2mm}

\begin{proof}
    Let us fix $I=(i_1,...,i_k), \, k \geq 1$. Applying Theorem \ref{C^alpha estimate for T_f phi, 3.2} we obtain
    \begin{equation*}
        ||T_{\Gamma_{\overline{x}}^{I;J}}(Y_J L_{\overline{x}}u)||_{C^{\alpha}_Y(B_R(\overline{x}))} \leq c\left(\mathbb{G}, \alpha, \Theta_{\Gamma_{\overline{x}}^{I;J},3}\right) |Y_J L_{\overline{x}}u|_{C^{\alpha}_Y(B_R(\overline{x}))}
    \end{equation*}
    for all $J=(j_1,...,j_k)$, where, recalling the definition of $T_f$ (see Theorem \ref{librone 6.32-6.33-8.24}): 
    \begin{align*}
        T_{\Gamma_{\overline{x}}^{I;J}}(Y_J L_{\overline{x}}u) = Y_h Y_l \int_{\mathbb{R}^N} \Gamma_{\overline{x}}^{I;J}(y^{-1}*x) \, Y_J L_{\overline{x}}u(y) \, dy.
    \end{align*}
    for arbitrary $1 \leq h,l \leq m$. Moreover, by Theorem \ref{rep formula higher derivs with Gamma^J}, it holds
    \begin{equation*}
        \Theta_{\Gamma_{\overline{x}}^{I;J}, \, 3} \leq c(\mathbb{G}, \Lambda,k),
    \end{equation*}
    whereas by means of Lemma \ref{C^alpha estimate for L_x-L} we have
    \begin{align*}
        & |Y_J L_{\overline{x}}u|_{C^{\alpha}_Y(B_R(\overline{x}))} \leq |Y_J Lu|_{C^{\alpha}_Y(B_R(\overline{x}))} + |Y_J (L_{\overline{x}}-L)u|_{C^{\alpha}_Y(B_R(\overline{x}))} \\
        & \leq |Y_J Lu|_{C^{\alpha}_Y(B_R(\overline{x}))} + c(\alpha, ||a||_k) \{ R^{\alpha} ||D^{k+2} u||_{C^{\alpha}_Y(B_R(\overline{x}))}+ ||u||_{C^{k+1,\alpha}_Y(B_R(\overline{x}))} \}
    \end{align*}
    Finally, thanks again to (\ref{rep formula Y_I u}), it is immediate that
    \begin{equation*}
        Y_h Y_l Y_{I} u(x) = \sum_{|J|=k} T_{\Gamma_{\overline{x}}^{I;J}}(Y_J L_{\overline{x}}u),
    \end{equation*}
    which yields
    \begin{align*}
        & ||Y_h Y_l Y_{I} u||_{C^{\alpha}_Y(B_R(\overline{x}))} \leq \sum_{|J|=k} ||T_{\Gamma_{\overline{x}}^{I;J}}(Y_J L_{\overline{x}}u)||_{C^{\alpha}_Y(B_R(\overline{x}))} \\ 
        & \leq c_1(k,\mathbb{G}, \Lambda, \alpha)  |Lu|_{C^{k,\alpha}_Y(B_R(\overline{x}))} + c_2(\mathbb{G}, \Lambda, \alpha, k, ||a||_k) \times  \\
        & \times \left\{ R^{\alpha}  ||D^{k+2} u||_{C^{\alpha}_Y(B_R(\overline{x}))}+ ||u||_{C^{k+1,\alpha}_Y(B_R(\overline{x}))} \right\}.
    \end{align*}
    Thus, summing over $h, \, l = 1, \, ..., \, m$ and over $I$ with $|I|=k$ we get
    \begin{align*}
        & ||D^{k+2} u||_{C^{\alpha}_Y(B_R(\overline{x}))} = \sum_{h,l=1}^m \sum_{|I|=k} ||Y_h Y_l Y_I u||_{C^{\alpha}_Y(B_R(\overline{x}))} \\
        & \leq c_{1,k}(\mathbb{G}, \Lambda, \alpha) |Lu|_{C^{k,\alpha}_Y(B_R(\overline{x}))} + c_{2,k}(\mathbb{G}, \Lambda, \alpha, ||a||_k) \times 
        \\
        & \times \left\{ R^{\alpha} ||D^{k+2} u||_{C^{\alpha}_Y(B_R(\overline{x}))} + ||u||_{C^{k+1,\alpha}_Y(B_R(\overline{x}))} \right\};
    \end{align*}
    so for all $0 < R \leq \overline{R}_k := \min\{1, \, (2c_{2,k})^{-1/\alpha}\}$ and $u \in C^{k+2,\alpha}_Y(B_R(\overline{x}))$ we have
    \begin{equation*}
        ||D^{k+2} u||_{C^{\alpha}_Y(B_R(\overline{x}))} \leq c(\mathbb{G}, \Lambda, \alpha, k, ||a||_k) \{ |Lu|_{C^{k,\alpha}_Y(B_R(\overline{x}))} + ||u||_{C^{k+1,\alpha}_Y(B_R(\overline{x}))} \}.
    \end{equation*}
\end{proof}

\vspace{2mm}

We now turn to the proof of Theorem \ref{global est higher derivs thm}, which we split in two parts: first we prove estimate (\ref{est high order derivs ineq}) assuming a priori that $u \in C^{k+2,\alpha}_Y(\mathbb{R}^N)$; later we show that if $u$ is just in $C^{2,\alpha}_Y(\mathbb{R}^N)$ but $Lu \in C^{k,\alpha}_Y(\mathbb{R}^N)$, then $u \in C^{k+2,\alpha}_Y(\mathbb{R}^N)$.

\vspace{2mm}

\begin{theorem}[Global a priori estimates]
\label{first part thm 1.6: estimates}
    Under the assumptions of Theorem \ref{global est higher derivs thm}, for any $k \geq 0$ there exists a constant $c_k=c(k,\mathbb{G}, \Lambda, \alpha, ||a||_k)>0$ such that for every $u \in C^{k+2,\alpha}_Y(\mathbb{R}^N)$ the following estimate holds:
    \begin{equation}
        \label{est higher derivs ineq} ||u||_{C^{k+2,\alpha}_Y(\mathbb{R}^N)} \leq c_k \left\{ ||Lu||_{C^{k,\alpha}_Y(\mathbb{R}^N)} + ||u||_{L^{\infty}(\mathbb{R}^N)} \right\}
    \end{equation}
\end{theorem}

\vspace{2mm}

\begin{proof}
We prove the statement by induction on $k \geq 0$, the case $k=0$ being given by Theorem \ref{1.4: Schauder est Carnot thm}. Fixed $k \geq 1$, assume that there holds
\begin{equation*}
    ||v||_{C^{k+1,\alpha}_Y(\mathbb{R}^N)} \leq c_{k-1} \left\{ ||Lv||_{C^{k-1,\alpha}_Y(\mathbb{R}^N)} + ||v||_{L^{\infty}(\mathbb{R}^N)} \right\} \quad \forall v \in C^{k+1,\alpha}_Y(\mathbb{R}^N).
\end{equation*}
Then, for any $u \in C^{k+2,\alpha}_Y(\mathbb{R}^N)$, we can split
    \begin{align}
        & \label{split of the k+2 Hölder norm}||u||_{C^{k+2,\alpha}_Y(\mathbb{R}^N)} = ||D^{k+2} u||_{C^{\alpha}_Y(\mathbb{R}^N)} + ||u||_{C^{k+1,\alpha}_Y(\mathbb{R}^N)}
    \end{align}
    and focus just on the former term at right-hand side, since the latter is already handled by inductive assumption. Let $0<R<\overline{R_k}/2$ with $\overline{R_k}$ as in the statement of Theorem \ref{local est high order der thm}, then let $\phi \in C^{\infty}_{0}(\mathbb{R}^N; [0,1])$ s.t. $\phi \equiv 1 \:$ in $B_R(0)$ and $\phi \equiv 0 \:$ outside $B_{2R}(0)$.
    Let also $\{B_R(\overline{x_i})\}_{i \in \mathbb{N}}$ be a covering of $\mathbb{R}^N$ and $\varphi_i=\phi(\overline{x_i}^{-1} * \cdot)$ for all $i \geq 1$. Applying Theorem \ref{local est high order der thm} to $u \varphi_i$ we get 
    \begin{align}
        & ||D^{k+2} u||_{C^{\alpha}_Y(B_R(\overline{x_i}))} \leq  ||D^{k+2}(u \varphi_i)||_{C^{\alpha}_Y(B_{2R}(\overline{x_i}))} \leq c \, \{||D^{k} (L(u \varphi_i))||_{C^{\alpha}_Y(B_{2R}(\overline{x_i}))} \nonumber\\
        & \label{control for C^k-1 norm u} + ||u \varphi_i||_{C^{k+1,\alpha}_Y(B_{2R}(\overline{x_i}))} \} \leq c \left\{||L(u \varphi_i)||_{C^{k,\alpha}_Y(B_{2R}(\overline{x_i}))} + ||u \varphi_i||_{C^{k+1,\alpha}_Y(B_{2R}(\overline{x_i}))} \right\}.
    \end{align}
    To shorten the notation, we denote $B_i:=B_{2R}(\overline{x_i})$ and $B_0:=B_{2R}(0)$. We have 
    \begin{align*}
        & L(u \varphi_i) = Lu \cdot \varphi_i + u \cdot L \varphi_i + 2 a^{hl} Y_h u \cdot Y_l \varphi_i
    \end{align*}
    and $Y_I \varphi_i=(Y_i \phi)(\overline{x_i}^{-1} * \cdot)$ for all $I$ by left-invariance of $\{Y_1,...,Y_m\}$, so
    \begin{align*}
        & ||\varphi_i||_{C^{h,\alpha}_Y(B_i)}=||\phi||_{C^{h,\alpha}_Y(B_0)} \quad \forall h \geq 0;\\  
        & ||L \varphi_i||_{C^{k,\alpha}_Y(B_i)} \leq ||a^{hl}||_{C^{k,\alpha}_Y(B_i)} ||Y_h Y_l \varphi_i||_{C^{k,\alpha}_Y(B_i)} \leq c(||a||_k) ||\phi||_{C^{k+2,\alpha}_Y(B_0)},
    \end{align*}
    having used also (\ref{Hölder norm of a product}). Plugging everything in (\ref{control for C^k-1 norm u})
    \begin{align*}
        & ||D^{k+2} u||_{C^{\alpha}_Y(B_R(\overline{x_i}))} \leq ||Lu||_{C^{k,\alpha}_Y(B_i)} ||\varphi_i||_{C^{k,\alpha}_Y(B_i)} + ||u||_{C^{k,\alpha}_Y(B_i)} ||L \varphi_i||_{C^{k,\alpha}_Y(B_i)} \\
        & + 2||a^{hl}||_{C^{k,\alpha}_Y(B_i)} ||Y_h u||_{C^{k,\alpha}_Y(B_i)} ||Y_l \varphi_i||_{C^{k,\alpha}_Y(B_i)} \\
        & \leq c \, ||\phi||_{C^{k+2,\alpha}_Y(\mathbb{R}^{N})} \left\{ ||Lu||_{C^{k,\alpha}_Y(\mathbb{R}^{N})} + ||u||_{C^{k+1,\alpha}_Y(\mathbb{R}^{N})} \right\}
    \end{align*}
    where by construction $c$ may depend on $k,\, \mathbb{G},\,  \Lambda, \, \alpha, \, ||a||_k$, but not on $i$. Since $\{B_R(\overline{x_i})\}_{i \in \mathbb{N}}$ is a covering of $\mathbb{R}^N$, taking the supremum on $i \geq 1$ we get
    \begin{equation*}
        ||D^{k+2} u||_{C^{\alpha}_Y(\mathbb{R}^N)} \leq c(k,\mathbb{G}, \Lambda, \alpha, ||a||_k) \left\{ ||Lu||_{C^{k,\alpha}_Y(\mathbb{R}^N)} + ||u||_{L^{\infty}(\mathbb{R}^N)} \right\}
    \end{equation*}
    and the thesis follows by (\ref{split of the k+2 Hölder norm}) and our inductive assumption.
\end{proof}

\vspace{2mm}

Finally, we prove the following:

\vspace{2mm}

\begin{theorem}[Global regularity result]
    \label{global regularity thm}
    Under the assumptions of Theorem \ref{1.4: Schauder est Carnot thm}, if $u \in C^{2,\alpha}_Y(\mathbb{R}^N)$ is such that $Lu \in C^{k,\alpha}_Y(\mathbb{R}^N), \, k \geq 1$, then $u \in C^{k+2,\alpha}_Y(\mathbb{R}^N)$.
\end{theorem}

\vspace{2mm}

\hspace{-7mm} The proof is strongly based on the following deep theorem (see \cite[Theorem 12.2.b]{librone}), which gives the local version of the same result:

\vspace{2mm}

\begin{theorem}
\label{local Hölder regularity thm}
    Let $Y_1, \, ..., \, Y_m$ be a system of Hörmander vector fields in a bounded domain $\Omega \subset \mathbb{R}^N$, let $\{a_{ij}\}_{1 \leq i,j \leq m}$ satisfy (H3)
    for some constant $\Lambda>0$ and $L$ be as in (\ref{our operator L_Y Carnot}). Then for every domain $\Omega' \Subset \Omega$, nonnegative integer $k$, $\alpha \in (0,1),$ if $ a_{ij} \in C^{k,\alpha}_Y(\mathbb{R}^N)$, for any $u \in C^{2,\alpha}_Y(\Omega)$ such that $Lu \in C^{k,\alpha}_Y(\Omega)$ it holds $u \in C^{k+2,\alpha}_Y(\Omega').$
\end{theorem}

\vspace{2mm}

\hspace{-7mm} For the sake of completeness, we point out that the cited theorem is much more general than what we actually need, since no underlying algebraic structure is required for the system $\{X_i\}_i$; moreover, the original theorem in \cite{librone} also gives a local estimate of $||u||_{C^{k+2,\alpha}_X}$ in terms of $||Lu||_{C^{k,\alpha}_Y}$ and $||u||_{L^{\infty}}$, which however we can not exploit since the dependence of the involved constant on the subdomain $\Omega'$ makes it unsuitable for our purposes. Instead, we will rely on the global estimate (\ref{est higher derivs ineq}) obtained above.

To prove Theorem \ref{global regularity thm}, we need also the following result about radial cutoff functions:

\vspace{2mm}

\begin{lemma}
\label{cutoff function lemma}
    For any $R \geq 1$ and any positive integer $h$, there exists a radial function $\phi_R \in C^{\infty}_{0}(\mathbb{R}^N; [0,1])$ such that $\phi_R \equiv 1$ in $B_R(0)$, $\phi_R \equiv 0$ outside $B_{2R}(0)$ and all derivatives up to order $h$ of $\phi_R$ along $Y_1,...,Y_m$ are bounded by a constant $c_h>0$ independent of $R$.
\end{lemma}

\vspace{2mm}

\begin{proof}
    We show that the following function does the job:
    \begin{equation*}
        \phi_R(x)= \begin{cases}
            1 \hspace{2cm} \text{if} \: ||x|| \leq R\\
            e^{\frac{1}{||x||-2R}+\frac{1}{R}} \hspace{5mm} \text{if} \: R<||x||<2R\\
            0 \hspace{2cm} \text{if} \: ||x|| \geq 2R.
        \end{cases}
    \end{equation*}
    Everything except the $R$-uniform boundedness of the derivatives is immediate.\\
    Let us prove that for any (non-empty) multindex $I$ there holds
    \begin{equation}
    \label{rep formula Y_I phi_R}
        Y_I \phi_R(x)= \phi_R(x) \sum_{2 \leq h \leq 2k} \, \sum_{1-k \leq p \leq 0} \frac{f_{h,p,I}(x)}{(||x||-2R)^h} \quad\: R<||x||<2R,
    \end{equation}
    for suitable functions $\{f_{h,p,I}\}_{h,p}$, each smooth in $\{R<||x||<2R\}$ and homogeneous of some nonpositive degree $p$. We proceed by induction on $k=|I| \geq 1$; if $k=1$, i.e. $Y_I=Y_i$ for some $1 \leq i \leq m$, we have (for $R<||x||<2R$)
    \begin{align}
    \label{explicit form Y_i phi_R}
        Y_i \phi_R(x)=\phi_R(x) Y_i\left(\frac{1}{||x||-2R} \right)=-\frac{\phi_R(x)}{(||x||-2R)^2} Y_i ||x||.
    \end{align}
    By definition (\ref{def homogeneous norm}) and $1$-homogeneity of $\{Y_i\}_i$, it is clear that $Y_i ||x||$ is homogeneous of degree $0$, therefore the thesis holds with $f_{2,0,i}(x):=-Y_i ||x||$.\\
Fixed $k \geq 1$, assume the thesis holds for every $|I|=k$ and consider a multindex $I'$ with $|I'|=k+1$, namely $I'=(i,I)$ for suitable $1 \leq i \leq m$ and $|I|=k$. Differentiating (\ref{rep formula Y_I phi_R}) along $Y_i$, by Leibniz rule and (\ref{explicit form Y_i phi_R}) we get
\begin{align*}
    & Y_{I'} \phi_R(x)= \sum_{2 \leq h \leq 2k, \, 1-k \leq p \leq 0} Y_i \left( \phi_R(x) \frac{f_{h,p,I}(x)}{(||x||-2R)^h} \right) = \\
    & \quad = \sum_{h,p} \left\{\frac{Y_i \phi_R(x) f_{h,p,I}(x)}{(||x||-2R)^h}+\frac{\phi_R(x) Y_i f_{h,p,I}(x)}{(||x||-2R)^h}-h\frac{\phi_R(x) f_{h,p,I}(x)}{(||x||-2R)^{h+1}} Y_i ||x|| \right\} \\
    & \quad = \sum_{h,p} \phi_R(x) \left[\frac{-Y_i (||x||)f_{h,p,I}(x)}{(||x||-2R)^{h+2}} + \frac{Y_i f_{h,p,I}(x)}{(||x||-2R)^h} - \frac{Y_i(||x||) f_{h,I}(x)}{(||x||-2R)^{h+1}} \right],
\end{align*}
thus (up to rearranging terms) the claim is achieved, provided the numerators of the three fractions above are homogeneous of nonpositive degrees $ \geq -k$ for all $2 \leq h \leq 2k, \, 1-k \leq p \leq 0$: this is immediate in view of the $0$-homogeneity of $Y_i ||x||$, in fact by inductive assumption $f_{h,p,I}$ is homogeneous of degree $p \geq 1-k$, and thus $Y_i f_{h,p,I}$ is homogeneous of degree $p-1 \geq -k$. This completes our proof of (\ref{rep formula Y_I phi_R}).
\\
Now, for all $1\leq R<||x||<2R$, the $p$-homogeneity of $f_{h,p,I}$ with $p \leq 0$ yields
\begin{align*}
    |f_{h,p,I'}(x)| = \left|f_{h,p,I'}\left( \frac{x}{||x||} \right)\right| \cdot ||x||^p \leq ||f_{h,p,I'}||_{L^{\infty}(B_1(0))}=c_{h,p,I'},
\end{align*}
namely $\{f_{h,p,I'}\}_{h,p}$ are all bounded independently of $R \geq 1$.
Finally, since the map $G_{\beta}(t):=t^{-\beta} e^{t}$ is bounded in $\mathbb{R}^+$ for any real $\beta$, then also the term
\begin{align*}
    \frac{\phi_R(x)}{(||x||-2R)^h}=e^{1/R} G_{h} \left( \frac{1}{||x||-2R} \right) \leq e ||G_h||_{L^{\infty}(\mathbb{R}^+)}
\end{align*}
is bounded independently of $R \geq 1$ for every $h$, so the thesis follows by (\ref{rep formula Y_I phi_R}).
\end{proof}

\vspace{2mm}

\begin{proof}[Proof of Theorem \ref{global regularity thm}]
    The proof follows the lines of \cite[Theorem 6.1]{Global Sobolev est with a_ij} and proceeds by induction on $k \geq 1$. Given $u \in C^{2,\alpha}_Y(\mathbb{R}^N)$ with $Lu \in C^{1,\alpha}_Y(\mathbb{R}^N)$, by Theorem \ref{local Hölder regularity thm} we have $u \in C^{3,\alpha}_{Y,loc}(\mathbb{R}^N)$. Let $R \geq 1$ be fixed and $\phi_R \in C^{\infty}_0(\mathbb{R}^N;[0,1])$ with $\phi_R \equiv 1$ in $B_R(0)$, $\phi_R \equiv 0$ outside $B_{2R}(0)$ and the derivatives up to order $4$ of $\phi_R$ along $Y_1,...,Y_m$ bounded independently of $R$ (its existence is given by Lemma \ref{cutoff function lemma}). Obviously $u\phi_R \in C^{3,\alpha}_{Y}(\mathbb{R}^N)$, therefore exploiting (\ref{est higher derivs ineq}) and proceeding as in the proof of Theorem \ref{first part thm 1.6: estimates} one gets
    \begin{align*}
        & ||u||_{C^{3,\alpha}_Y(B_R(0))} \leq ||u \phi_R||_{C^{3,\alpha}_Y(\mathbb{R}^N)} \leq c \left\{ ||L(u \phi_R)||_{C^{1,\alpha}_Y(\mathbb{R}^N)} + ||u \phi_R||_{L^{\infty}(\mathbb{R}^N)} \right\} \\
        & \quad \leq c \, ||\phi_R||_{C^{3,\alpha}_Y(\mathbb{R}^N)} \left\{ ||Lu||_{C^{1,\alpha}_Y(\mathbb{R}^N)} + ||u||_{C^{2,\alpha}_Y(\mathbb{R}^N)} \right\} \\
        & \quad \text{by Theorem \ref{1.4: Schauder est Carnot thm}} \\
        & \quad \leq c \sup_{R \geq 1} ||\phi_R||_{C^{3,\alpha}_Y(\mathbb{R}^N)} \left\{||Lu||_{C^{1,\alpha}_Y(\mathbb{R}^N)} + ||u||_{L^{\infty}(\mathbb{R}^N)} \right\},
    \end{align*}
    with $c>0$ depending just on $\mathbb{G}, \, \Lambda, \, \alpha, \, ||a||_1$. We point out that the supremum at right-hand side is finite, since the bounds on the derivatives of $\phi_R$ hold uniformly w.r.t. $R \geq 1$; then taking the limit as $R \to +\infty$:
    \begin{equation*}
        ||u||_{C^{3,\alpha}_Y(\mathbb{R}^N)} \leq c(\mathbb{G}, \Lambda, \alpha, ||a||_1) \left\{ ||Lu||_{C^{1,\alpha}_Y(\mathbb{R}^N)} + ||u||_{L^{\infty}(\mathbb{R}^N)} \right\},
    \end{equation*}
    thus $u \in C^{3,\alpha}_Y(\mathbb{R}^N)$. Given $k \geq 2$, assume the thesis holds for all $v \in C^{2,\alpha}_Y(\mathbb{R}^N)$ with $Lv \in C^{k-1,\alpha}_Y(\mathbb{R}^N)$, then let $u\in C^{2,\alpha}_Y(\mathbb{R}^N)$ be such that $Lu \in C^{k,\alpha}_Y(\mathbb{R}^N)$; which yields $u \in C^{k+2,\alpha}_{Y,loc}(\mathbb{R}^N)$ by Theorem \ref{local Hölder regularity thm}. Let $\phi_R$ be a cutoff function as above, with all derivatives up to order $k+3$ bounded independently of $R$. Then $u \phi_R \in C^{k+2,\alpha}_Y(\mathbb{R}^N)$; therefore, analogously as in the case $k=1$:
    \begin{align*}
        & ||u||_{C^{k+2,\alpha}_Y(B_R(0))} \leq ||u \phi_R||_{C^{k+2,\alpha}_Y(\mathbb{R}^N)} \leq c_k \, \left\{ ||L(u \phi_R)||_{C^{k+2,\alpha}_Y(\mathbb{R}^N)} + ||u||_{L^{\infty}(\mathbb{R}^N)} 
        \right\} \\
        & \leq c_k \sup_{R \geq 1} ||u \phi_R||_{C^{k+2,\alpha}_Y(\mathbb{R}^N)} \left\{ ||Lu||_{C^{k,\alpha}_Y(\mathbb{R}^N)} + ||u||_{C^{k+1, \alpha}_Y(\mathbb{R}^N)} \right\}.
    \end{align*}
    The supremum is finite since the bounds on derivatives of $\phi_R$ are uniform in $R \geq 1$; moreover $u \in C^{k+1,\alpha}_Y(\mathbb{R}^N)$ by inductive assumption, so the whole right-hand side is finite (and independent of $R$). Taking the limit as $R \to +\infty$:
    \begin{equation*}
        ||u||_{C^{k+2,\alpha}_Y(\mathbb{R}^N)} \leq c_k \, \left\{ ||Lu||_{C^{k,\alpha}_Y(\mathbb{R}^N)} + ||u||_{C^{k+1,\alpha}_Y(\mathbb{R}^N)} \right\}<+\infty,
    \end{equation*}
    thus $u \in C^{k+2,\alpha}_Y(\mathbb{R}^N)$.
\end{proof}

\vspace{2mm}

\hspace{-7mm} Combining Theorems \ref{first part thm 1.6: estimates} and \ref{global regularity thm} we readily get Theorem \ref{global est higher derivs thm}.

\vspace{5mm}

\section{Removal of the group structure} \label{sec 5}
Let us return to our original framework, with $X_1, \, ..., \, X_m$ homogeneous (but not left-invariant) Hörmander vector fields on $(\mathbb{R}^n, \delta_{\lambda})$ and $A=A(x)$ satisfying (H3) in $\mathbb{R}^n$ for some constants $\Lambda \in (0,1), \, \alpha \in (0,1), \, ||a||_{0,X}>0$; let also $\widetilde{X}_1, \, ..., \, \widetilde{X}_m$ and $\widetilde{A}=\widetilde{A}(x,\xi)$ be their lifted versions, in the sense of Theorem \ref{global BB lifting thm}.\\
Since $\widetilde{X}_1, \, ..., \, \widetilde{X}_m$ are Lie generators of a Carnot group $\mathbb{G}$ on $\mathbb{R}^N$, the theory developed in Sections \ref{sec 2}-\ref{sec 4} applies to the operator $\widetilde{\mathcal{L}}=\widetilde{a}^{ij}(\cdot) \, \widetilde{X}_i \widetilde{X}_j$, namely there is a constant $c_0=c_0(\mathbb{G}, \Lambda, \alpha, ||a||_0)>0$ s.t. for any $v\in C^{2,\alpha}_{\widetilde{X}}(\mathbb{R}^N)$ there holds
\begin{equation}
    \label{lifted Schauder est} ||v||_{C^{2,\alpha}_{\widetilde{X}}(\mathbb{R}^N)} \leq c_0 \left\{ ||\widetilde{\mathcal{L}}v||_{C^{\alpha}_{\widetilde{X}}(\mathbb{R}^N)} + ||v||_{L^{\infty}(\mathbb{R}^N)} \right\}
\end{equation}
and, if moreover $\widetilde{a}_{ij} \in C^{k,\alpha}_{\widetilde{X}}(\mathbb{R}^N)$, $k \geq 1$ for all $1 \leq i,j \leq m$, then there exists a constant $c_k=c(k,\, \mathbb{G}, \, \Lambda, \, \alpha, \, ||\widetilde{a}||_k)>0$ such that for any $v \in C^{2,\alpha}_{\widetilde{X}}(\mathbb{R}^N)$ with $\mathcal{\widetilde{L}}v \in C^{k,\alpha}_{\widetilde{X}}(\mathbb{R}^N)$ there holds $v \in C^{k+2,\alpha}_{\widetilde{X}}(\mathbb{R}^N)$ and
\begin{equation*}
||v||_{C^{k+2,\alpha}_{\widetilde{X}}(\mathbb{R}^N)} \leq c_k \left\{ ||\widetilde{\mathcal{L}}v||_{C^{k,\alpha}_{\widetilde{X}}(\mathbb{R}^N)} + ||v||_{L^{\infty}(\mathbb{R}^N)} \right\}.
\end{equation*}
\\
Denoted by $d_X$ and $d_{\widetilde{X}}$ the CC distances induced by $X_1, \, ..., \, X_m$ on $\mathbb{R}^n$ and by $\widetilde{X}_1, \, ..., \, \widetilde{X}_m$ on $\mathbb{R}^N$ (respectively), in order to deduce Theorem \ref{global C^alpha estimate thm} from (\ref{lifted Schauder est}) we need a couple of results (Lemma \ref{f C^alpha -> f tilde C^alpha} and Lemma \ref{f tilde C^alpha -> f C^alpha} below) relating Hölder seminorms in the original space and in the lifted space. \\In the following, we denote by $E:C(\mathbb{R}^n) \to C(\mathbb{R}^N)$ the \say{extension} operator $(Ef)(x,\xi):=f(x)$. We point out that, although in principle $X_i$ ($1 \leq i \leq m$) is a vector field on $\mathbb{R}^n$, we can equivalently think of it as a vector field on $\mathbb{R}^N$ which does not act on the variables $\xi_1, \, ..., \, \xi_p$; with this notation, it is immediate that $X_i$ and $E$ commute, namely $\widetilde{X}_i(Ef)=E(X_i f)$ for all $1 \leq i \leq m$, and recursively $\widetilde{X}_I(Ef)=E(X_I f)$ for any multindex $I$ (for all $f$ smooth enough).

\vspace{2mm}

\begin{lemma} 
\label{f C^alpha -> f tilde C^alpha}
If $f \in C^{\alpha}_X(\mathbb{R}^n)$, then $\widetilde{f}:=Ef \in C^{\alpha}_{\widetilde{X}}(\mathbb{R}^N)$ and 
\begin{align*}||\widetilde{f}||_{L^{\infty}(\mathbb{R}^N)} = ||f||_{L^{\infty}(\mathbb{R}^n)}, \quad |\widetilde{f}|_{C^{\alpha}_{\widetilde{X}}(\mathbb{R}^N)} \leq |f|_{C^{\alpha}_X(\mathbb{R}^n)} \nonumber.
\end{align*}

\end{lemma}

\vspace{2mm}

For the proof see \cite[Proposition 10.39]{librone}.

\vspace{2mm}

\begin{lemma}
\label{f tilde C^alpha -> f C^alpha}
    If a continuous function $f : \mathbb{R}^n \to \mathbb{R}$ is such that $\widetilde{f} \in C^{\alpha}_{\widetilde{X}}(\mathbb{R}^N)$, then $f \in C^{\alpha}_X(\mathbb{R}^n)$ and there exists a constant $c>0$ (independent of $f$) such that 
    \begin{equation*}
    ||f||_{C^{\alpha}_X(\mathbb{R}^n)} \leq c \, ||\widetilde{f}||_{C^{\alpha}_{\widetilde{X}}(\mathbb{R}^N)}.
    \end{equation*}
\end{lemma}

\vspace{2mm}

\begin{proof}
    For any fixed $R>0$ and $(\overline{x}, \overline{\xi}) \in \mathbb{R}^N$ we have (cf \cite[Proposition 8.3]{Hölder spaces: original vs lifted})
    \begin{equation*}
        |f|_{C^{\alpha}_X(B_R(\overline{x}))} \leq c \, |\widetilde{f}|_{C^{\alpha}_{\widetilde{X}}(B_R(\overline{x}, \overline{\xi}))} \leq c \, |\widetilde{f}|_{C^{\alpha}_{\widetilde{X}}(\mathbb{R}^N)}
    \end{equation*}
    for some constant $c>0$ that does not depend on $f$ nor $R$ nor $(\overline{x}, \overline{\xi})$). Then, given a covering $\{B_R(\overline{x_i})\}_{i \in \mathbb{N}}$ of $\mathbb{R}^n$, there holds
    \begin{equation*}
        |f|_{C^{\alpha}_X(B_R(\overline{x_i}))} \leq c \, |\widetilde{f}|_{C^{\alpha}_{\widetilde{X}}(\mathbb{R}^N)} \hspace{5mm} \forall i \in \mathbb{N}
    \end{equation*}
    with $c$ independent of $i$ by construction, which clearly yields
    \begin{equation*}
        ||f||_{C^{\alpha}_X(B_R(\overline{x_i}))} \leq c \, ||\widetilde{f}||_{C^{\alpha}_{\widetilde{X}}(\mathbb{R}^N)} \hspace{5mm} \forall i \in \mathbb{N}
    \end{equation*}
    (with $c$ independent on $i$). Taking the supremum on $i$, we get
    \begin{equation*}
        ||f||_{C^{\alpha}_X(\mathbb{R}^n)} \leq c \sup_{i \in \mathbb{N}} ||f||_{C^{\alpha}_X(B_R(\overline{x_i}))} \leq c' \, ||\widetilde{f}||_{C^{\alpha}_{\widetilde{X}}(\mathbb{R}^N)}.
    \end{equation*}
\end{proof}

\vspace{2mm}

\begin{lemma}
\label{lemma 5.3 final}
    Given $u \in C^{k+2,\alpha}_X(\mathbb{R}^n)$, then $\widetilde{u}:=E(u) \in C^{k+2,\alpha}_{\widetilde{X}}(\mathbb{R}^N)$ and
    \begin{align}
        & \label{||u||_2,alpha <= c||f tilde||2,alpha} ||u||_{C^{k+2,\alpha}_X(\mathbb{R}^n)} \leq c \, ||\widetilde{u}||_{C^{k+2,\alpha}_{\widetilde{X}}(\mathbb{R}^N)},
        \\
        & \label{||Lu tilde|| <= ||Lu||}||\mathcal{\widetilde{L}} \widetilde{u}||_{C^{k,\alpha}_{\widetilde{X}}(\mathbb{R^N)}} \leq ||\mathcal{L} u||_{C^{k,\alpha}_X(\mathbb{R}^n)},
    \end{align}
    where $c>0$ is the constant of Lemma \ref{f tilde C^alpha -> f C^alpha}.
\end{lemma}

\vspace{2mm}

\begin{proof}
    Since $\widetilde{X}_I \widetilde{u}=E(X_Iu)$ for all $|I| \leq k+2$, by Lemma \ref{f C^alpha -> f tilde C^alpha} we have $\widetilde{u} \in C^{k+2,\alpha}_Y(\mathbb{R}^N)$; hence Lemma \ref{f tilde C^alpha -> f C^alpha} applies for any $Y_I u, \, |I| \leq k+2$, yielding
    \begin{align*}
        & ||u||_{C^{k+2,\alpha}_X(\mathbb{R}^n)} = \sum_{|I| \leq k+2} ||X_I u||_{C^{\alpha}_{X}(\mathbb{R}^n)} \\
        & \quad \leq c\sum_{|I| \leq k+2} ||\widetilde{X}_I \widetilde{u}||_{C^{\alpha}_{\widetilde{X}}(\mathbb{R}^N)} = c \, ||\widetilde{u}||_{C^{k+2,\alpha}_{\widetilde{X}}(\mathbb{R}^N)}.
    \end{align*}
    Since $\widetilde{a}_{ij}:=E(a_{ij})$ then the commutation of $\{Y_i\}_i$ with the extension $E$ gives also $\mathcal{\widetilde{L}} \widetilde{u} = E(\mathcal{L} u)$, so by Lemma \ref{f C^alpha -> f tilde C^alpha} we immediately get (\ref{||Lu tilde|| <= ||Lu||}).
\end{proof}

\vspace{2mm}

\hspace{-7mm} Finally:

\vspace{2mm}

\begin{proof}[Proof of Theorem \ref{global C^alpha estimate thm}]
    Given $k \geq 0$, let $u \in C^{2,\alpha}_X(\mathbb{R}^n)$, let $\widetilde{u}=E(u)$ and let us assume that $a_{ij} \in C^{k,\alpha}_X(\mathbb{R}^n)$ for all $1 \leq i,j \leq m$ and $\mathcal{L} u \in C^{k,\alpha}_X(\mathbb{R}^n)$ (which automatically holds for $k=0$, in view of (H3)). For $k \geq 1$, by Lemma \ref{f C^alpha -> f tilde C^alpha} we have $\widetilde{a}_{ij} \in C^{k,\alpha}_{\widetilde{X}}(\mathbb{R}^N)$, $\widetilde{u} \in C^{2,\alpha}_{\widetilde{X}}(\mathbb{R}^N)$ and $\mathcal{\widetilde{L}} \widetilde{u} \in C^{k,\alpha}_Y(\mathbb{R}^N)$; so $\widetilde{u} \in C^{k+2,\alpha}_{\widetilde{X}}(\mathbb{R}^N)$ and (by either Theorem \ref{1.4: Schauder est Carnot thm} if $k=0$ or \ref{global est higher derivs thm} if $k \geq 1$) we have
    \begin{equation}
        \label{Schauder est u tilde}||\widetilde{u}||_{C^{k+2,\alpha}_{\widetilde{X}}(\mathbb{R}^N)} \leq c(k,\mathbb{G}, \Lambda, \alpha, ||a||_k) \left\{ ||\mathcal{\widetilde{L}} \widetilde{u}||_{C^{k,\alpha}_{\widetilde{X}}(\mathbb{R}^N)} + ||\widetilde{u}||_{L^{\infty}(\mathbb{R}^N)} \right\}.
    \end{equation}
    Then, in view of Lemma \ref{lemma 5.3 final}, we get $u \in C^{k+2,\alpha}_X(\mathbb{R}^n)$ and
    \begin{align*}
        & ||u||_{C^{k+2,\alpha}_X(\mathbb{R}^n)} \leq ||\widetilde{u}||_{C^{k+2,\alpha}_{\widetilde{X}}(\mathbb{R}^N)} \leq c(k,\mathbb{G}, \Lambda, \alpha, ||a||_k) \{ ||\mathcal{\widetilde{L}} \widetilde{u}||_{C^{k,\alpha}_{\widetilde{X}}(\mathbb{R}^N)} + \\
        & + ||\widetilde{u}||_{L^{\infty}(\mathbb{R}^N)} \}\leq c'(k,\mathbb{G}, \Lambda, \alpha, ||a||_k) \left\{ ||\mathcal{L} u||_{C^{k,\alpha}_X(\mathbb{R}^n)} + ||u||_{L^{\infty}(\mathbb{R}^n)} \right\}.
    \end{align*}
\end{proof}

\vspace{1cm}

\begin{address}
    M. Faini: Dipartimento di Matematica, Politecnico di Milano, via Bonardi 9, 20133 Milano, Italy.\\
    email address: matteo.faini@polimi.it
\end{address}

\end{document}